\documentclass[preprint,12pt]{elsarticle}
\usepackage{amssymb}
\usepackage{amsmath}
\usepackage{lipsum}
\usepackage{amsfonts}
\usepackage{graphicx}
\usepackage{epstopdf}
\usepackage{algorithmic}
\usepackage{enumerate}
\usepackage{xspace}
\usepackage{bold-extra}
\usepackage[most]{tcolorbox}
\usepackage[utf8]{inputenc}
\usepackage[T1]{fontenc}

\usepackage{enumitem} %allows labeling the items of enumerations

\DeclareMathOperator{\rank}{rank}
\DeclareMathOperator{\trace}{tr}

\DeclareMathOperator{\extern}{ext}
\DeclareMathOperator{\dg}{diag}
\newcommand{\R}{\mathbb{R}}
\newcommand{\N}{\mathbb{N}}

\newtheorem{theorem}{Theorem}
\newtheorem{lemma}[theorem]{Lemma}
\newtheorem{proposition}[theorem]{Proposition}
\newtheorem{corollary}[theorem]{Corollary}
\newdefinition{remark}[theorem]{Remark}
\newproof{proof}{Proof}

\journal{European Journal of Control}

\begin{document}

\begin{frontmatter}



\title{Infinite-dimensional port-Hamiltonian systems with a stationary interface}


\author[inst1]{Alexander Kilian}

\affiliation[inst1]{organization={Faculty of Computer Science and Mathematics, University of Passau},%Department and Organization
            %addressline={Address One}, 
            %city={City One},
            %postcode={00000}, 
            %state={State One},
            country={Germany}}

\author[inst2]{Bernhard Maschke}
\author[inst3]{Andrii Mironchenko}
\author[inst1]{Fabian Wirth}

\affiliation[inst2]{organization={LAGEPP,  UMR CNRS 5007, Université Claude Bernard Lyon-1} ,%Department and Organization
            addressline={Bd du 11 Novembre 1918}, 
            city={Villeurbanne cedex},
            postcode={F-69622}, 
            %state={State Two},
            country={France}}

\affiliation[inst3]{organization={Department of Mathematics, University of Bayreuth},%Department and Organization
            %addressline={Address One}, 
            %city={City One},
            %postcode={00000}, 
            %state={State One},
            country={Germany}}

\begin{abstract}
We consider two systems of two conservation laws that are defined on complementary, one-dimensional spatial intervals and coupled by an interface as a single port-Hamiltonian system. In case of a fixed interface position, we characterize the boundary and interface conditions for which the associated port-Hamiltonian operator generates a contraction semigroup. Furthermore, we present sufficient conditions for the exponential stability of the generated $C_0$-semigroup. The results are illustrated by the example of two acoustic waveguides coupled by a membrane interface.
\end{abstract}




\end{frontmatter}


\section{Introduction}
Linear infinite-dimensional systems theory is a well-established subject \cite{Curtain2020,Tucsnak2009,Jacob2012}, constituting a basis for analysis of broad classes of distributed parameter systems. However, \emph{physical systems} belonging to mechanical, electric, hydraulic, thermal, and other domains share common specific structures that can be used to develop methods for modeling, analysis and control, which are more natural and easy-to-use than those available for general linear systems. 
The port-Hamiltonian systems theory \cite{Jacob2012, Duindam2009, Schaft2014} is a systematic framework that develops such methods.
This modeling framework is an extension of the \emph{Hamiltonian systems}, stemming from mechanics \cite{Arnold1978}, to open systems in a control systems' perspective where the interaction with the environment occurs via a pair of external variables called \emph{port variables} \cite{schaftGeomPhys02, Gorrec2005}.

Following the semigroup approach, consider the abstract Cauchy problem of the form
\begin{align*}
	\dot{x}(t) &= Ax(t),
 \quad t > 0, \\
	x(0) &= x_0 \in X,
\end{align*}
where $A \colon D(A) \subset X \to X$ 
is a linear unbounded operator on some Hilbert space $X$. In the port-Hamiltonian framework \cite[pp. 226-228]{Duindam2009}, \cite[Section 3]{Gorrec2005}, \cite[Chapter 2]{Villegas2007},  \cite[Chapter 3]{Augner2016}, \cite{Rashad_IMA18_20yearsBPHS},  the operator $A$ has the form
\begin{align*}
 A = \mathcal{J}(\mathcal{Q}\cdot),  
\end{align*}
where $\mathcal{J}$ is a Hamiltonian operator \cite{Olver93}, which is a formally skew-symmetric constant coefficient matrix differential operator, modeling the interdomain coupling in physical systems, and $\mathcal{Q}$ is a symmetric coercive and bounded multiplication operator defining the energy of the physical system.
Furthermore, the system is augmented with a pair of \emph{boundary port variables} defined by a boundary port operator $R_{\extern}$. This operator is associated with the Hamiltonian operator $\mathcal{J}$ \cite[Def. 3.5]{Gorrec2005}, in a very similar way as for boundary triplets \cite{Kurula_IJC_15}
\begin{align*}
    \begin{bmatrix}
				f_\partial \\
				e_\partial
			\end{bmatrix}= R_{\extern} \trace\left(\mathcal{Q}x\right),
\end{align*}
where $\trace$ denotes the trace operator.
The well-documented theory of strongly continuous semigroups \cite{Curtain2020, Engel2000, Pazy1983} may then be used to derive results on well-posedness, stability and stabilization by boundary feedback \cite[Section 4.4]{Rashad_IMA18_20yearsBPHS}. The most important class of semigroups in the framework of infinite-dimensional port-Hamiltonian systems are contraction semigroups, which are closely related to energy dissipation. Typically, dissipation phenomena are incorporated into the domain of the operator $A$ by means of the port variables, and this operator is the infinitesimal generator of a contraction semigroup. For control purposes, control and conjugated observation variables may be defined as a linear combination of the port variables, leading to the definition of (port-Hamiltonian) boundary control systems \cite[Chapter 10]{Tucsnak2009}, \cite{Fattorini1968}. Such systems have been extensively studied in \cite{Villegas2007, Augner2016}.

An important property of port-Hamiltonian systems is that the power-conserving interconnection of port-Hamiltonian systems again defines a port-Hamiltonian system. This interconnection is defined by the composition of the underlying (Stokes-)Dirac structures \cite[Section 2.2, Section 14.1]{Schaft2014} through a number of mutually shared port variables. The internally stored energy of the interconnected system, described by the Hamiltonian on the product space, is simply the sum of the Hamiltonians of the constituting subsystems. We refer to \cite[Chapter 6]{Schaft2014} for the finite-dimensional case as well as to \cite[Chapter 7]{Villegas2007}, \cite{Hamroun2010}, and the references therein for the infinite-dimensional case. In \cite[Section 2.2]{Diagne2013}, it has been depicted how port-Hamiltonian systems are coupled through their (fixed) boundaries. 

In this paper, we analyze systems with an interface within their spatial domain, which may be moving due to non-equilibrium conditions at the interface. Such systems arise, for instance, in varying causality systems, such as extrusion processes, where some parts are completely filled and some others are not \cite{Lotero_MCMDS17}, phase transition phenomena in multiphase systems, such as evaporators or condensers \cite{Willatzen_IntJRefrigeration1998}, or the celebrated Stefan problem \cite{Gupta2003}.

In \cite{Diagne2013}, the authors proposed the port-Hamiltonian formulation of a physical system containing a moving interface. Instead of considering coupling two systems by a Dirac structure or a port-Hamiltonian system \cite{Kurula_JMAA_10,Jacob_JEvEq2015c}, in \cite{Diagne2013}, one introduces the interface using additional variables, the \emph{color functions} \cite{Godlewski2004,Boutin2008}. These color functions are the characteristic functions of spatial domains separated by the interface and correspond to the so-called thin interface assumption. A motivation for this perspective is also provided by the phase field theory for multiphase systems, where the interface separating the different phases has a certain thickness and is characterized by a distributed variable called \emph{phase field} \cite{Vincent_IFAC_WC20_PhaseFields}. 

\textbf{Contribution.} Having proposed a port-Hamiltonian formulation for systems with an interface, the authors of \cite{Diagne2013} do not proceed to its well-posedness analysis. In this paper, motivated by \cite{Diagne2013}, we propose a mathematically rigorous definition of this class of systems, where the Hamiltonian operator depends on the color functions. Then, we model such systems as abstract evolution equations and use the methods of semigroup theory to provide easy-to-check criteria for boundary and interface conditions that guarantee that the system with a \emph{fixed} interface induces a contraction semigroup. Under somewhat stronger conditions, we prove the exponential stability of such a system. We apply our findings to the interface coupling of two acoustic waveguides. This work can be considered as a first step to the analysis of one-dimensional port-Hamiltonian systems with a moving interface, where the presented fixed interface model serves as a reference frame.

We note that the coupling problem we formulate here can also be embedded in a more general framework for the interconnection of port-Hamiltonian systems, which was developed in \cite{Augner2020}. Several of our results can be obtained using the methods presented in this reference. We see the following advantages in following the approach that we present here: The modeling presented here is closer to the actual problem and does not take a route through a significant body of abstract theory, even though we have to admit that the level of abstraction in the present paper is already significant. Secondly, the conditions we derive e.g. for exponential stability are more straightforward to formulate and to check than following the framework of \cite{Augner2020}. In addition, because we are closer to one very specific situation, our conditions are a bit less stringent than those derived for the much more general case, see the discussion after Theorem~\ref{Theorem Exponential Stability}.

	\textbf{Structure of the paper.} 
 Section \ref{Section Port-Hamiltonian Systems with Interface} introduces port-Hamiltonian systems with an interface. It is split into three parts. In Subsection \ref{Subsection Model Description}, we consider two systems of conservation laws defined on two complementary spatial domains that we want to interconnect at the interface position, which is assumed to be fixed. The notions of interface port variables and color functions are introduced, which allows us to formalize the state variables of the aggregate system on the composed domain. In Subsection \ref{Subsection Derivation of the Balance Equations}, we briefly describe the derivation of the conservation laws of the aggregate system having regard to the interface conditions. Subsection \ref{Subsection Port-Hamiltonian Formulation} is devoted to the port-Hamiltonian formulation of the system dynamics. We introduce the formally skew-symmetric differential operator associated with the aggregate system. Furthermore, we define the boundary and interface port variables in case of a fixed interface position, and state an energy balance equation for this class of systems. In Section~\ref{Section Generation of a Contraction Semigroup}, we characterize boundary and interface conditions for which the port-Hamiltonian operator associated with the considered port-Hamiltonian system generates a contraction semigroup on the energy space. Under stronger assumptions concerning the boundary conditions, we guarantee exponential stability of the semigroup, as we work out in Section \ref{Section Exponential Stability}.  
    In Section~\ref{sec:Example}, we apply our findings to the interface coupling of two acoustic waveguides. Section~\ref{Section Port-Hamiltonian systems with multiple interfaces} briefly discusses port-Hamiltonian systems with multiple stationary interfaces.
    In Section~\ref{sec:conclusions}, we conclude our paper. In particular, we discuss how our methods can be applied for the analysis of two port-Hamiltonian systems coupled via a moving interface, see also \cite{KMMW23} for more details.

    \paragraph*{Notation}
	Throughout the paper $X$ is a Hilbert space and $\mathcal{L}(X)$ is the space of linear bounded operators on $X$.
	We denote by $\mathcal{C}([a,b], X)$ and $\mathcal{C}^k([a,b], X)$ the vector spaces of $X$-valued, continuous and $k$-times continuously differentiable functions $f \colon [a,b] \to X$; $L^2([a,b], \mathbb{R}^n)$ denotes the vector space of $\mathbb{R}^n$-valued square-integrable functions on the interval $[a,b]$, and $L^{\infty}([a,b], \mathbb{R}^n)$ is the vector space of $\mathbb{R}^n$-valued Lebesgue measurable essentially bounded functions on $[a,b]$. Moreover, $H^1([a,b], \mathbb{R}^n)$ denotes the Sobolev space of functions from $L^2([a,b], \mathbb{R}^n)$ that have weak derivatives belonging to $L^2([a,b], \mathbb{R}^n)$. Further, the vector space of infinitely differentiable functions with compact support in $(a,b) \subset \mathbb{R}$, that is, the space of test functions on $(a,b)$, is denoted by $\mathcal{D}(a,b)$. The Euclidean norm in $\R^n$ is denoted by $|\cdot|$, and $I_n$ is the identity matrix of size $n\times n$.

\section{Port-Hamiltonian Systems with Interface}
\label{Section Port-Hamiltonian Systems with Interface}
In this section, we propose a port-Hamiltonian formulation for a system composed of two conservation laws, coupled by means of the interface, satisfying certain interface relations. For a detailed discussion, we refer to \cite{Diagne2013}, \cite[Section 5.2]{Kilian2022}.

\subsection{Model Description}
\label{Subsection Model Description}
 Let $a<0<b$, and let $l \in (a,b)$ be the fixed position of the interface. Consider two systems of conservation laws of the form
\begin{align}
	\begin{split}
		\frac{\partial}{\partial t} x^-(z,t) &= P_1 \frac{\partial}{\partial z} \big(\mathcal{Q}^-(z) x^-(z,t)\big), \quad  z \in [a,l), \quad t >0,  \\
		\frac{\partial}{\partial t} x^+(z,t) &= P_1 \frac{\partial}{\partial z} 
		\big(\mathcal{Q}^+(z) x^+(z,t)\big), \quad z \in (l,b], \quad t > 0,
	\end{split} 
	\label{Systems of Conservation Laws} 
\end{align}
defined on the respective interval $[a,l)$ or $(l,b]$, where
\begin{equation}
    P_1 := \begin{bmatrix}
        0 & -1 \\
        -1 & 0
    \end{bmatrix} \in \mathbb{R}^{2 \times 2},
    \label{Matrix P1}
\end{equation}
and $\mathcal{Q}^{\pm} \in L^{\infty}([a,b], \mathbb{R}^{2 \times 2})$ are symmetric and satisfy $mI_2 \leq \mathcal{Q}^{\pm}(z) \leq MI_2$ for almost all $z \in [a,b]$ and for some $0 < m \leq M$.

Let us define the \emph{flux variables}
\begin{align*}
	\mathcal{N}^-(x^-) &:= - P_1 (\mathcal{Q}^- x^-), \quad x^- \in X^- := L^2((a,l), \mathbb{R}^2), \\
	\mathcal{N}^+(x^+) &:= - P_1 (\mathcal{Q}^+ x^+), \quad x^+ \in X^+ := L^2((l,b), \mathbb{R}^2).
\end{align*} 
Since $\mathcal{Q}^{\pm}$ are essentially bounded, we have $\mathcal{N}^{\pm}(x^{\pm}) \in X^{\pm}$.

Our objective is to interconnect these two systems to a single port-Hamiltonian system on the composed spatial domain $[a,b]$, where the interface has the fixed position $l \in (a,b)$. To do so, we need to introduce suitable state variables defined on the composed domain, and we need to define some \emph{interface relations} at $z = l$. The interface relations simultaneously constitute the \emph{interface port variables} $(e_I, f_I) \in \mathbb{R}^2$. We model them by
\begin{align}
	f_I &:= \mathcal{N}_1^+(x^+)(l^+) = \mathcal{N}_1^-(x^-)(l^-), \label{Interface Variable fI} \\
	-e_I &:= \mathcal{N}_2^+(x^+)(l^+) - \mathcal{N}_2^-(x^-)(l^-). \label{Interface Variable eI}
\end{align}

Equation \eqref{Interface Variable fI} is a continuity equation of the flux variables $\mathcal{N}_1^{\pm}(x^{\pm})$, and \eqref{Interface Variable eI} is a balance equation of the flux variables $\mathcal{N}_2^{\pm}(x^{\pm})$. These interface relations are commonly considered (cf.  \cite{Godlewski2004, Boutin2008}) and are applicable to stationary and moving interface situations. For example, consider a power transmission line with an ohmic resistor placed at the interface position. In this case, $f_I$ requires the current to flow continuously through the resistor and $e_I$ quantifies the voltage drop across it. Similarly, if two isentropic gases are coupled at their interface by some piston in motion, then the interface relation \eqref{Interface Variable fI} corresponds to the requirement that the velocities of the gases on both sides of the piston are equal to the velocity of the piston. The interface relation \eqref{Interface Variable eI} corresponds to the balance of forces exerted by the gases on both sides, where $e_I$ quantifies the pressure jump, see \cite[Example 8]{Diagne2013}. 

Next, we define the \emph{color functions}, which are characteristic functions of the respective complementary subdomain of the considered spatial domain $[a,b]$, and are used to keep track of the interface position:
	\begin{equation}
		c_l^-(z) := \begin{cases} 
			1, & z \in [a,l), \\
			0, & z \in [l,b],
		\end{cases} \quad \text{and} \quad c_l^+(z) :=  \begin{cases} 
			0, & z \in [a,l], \\
			1, & z \in (l,b].
		\end{cases}
		\label{Characteristic Functions c- and c+}
	\end{equation}
Since, in this paper, we assume that the interface position is fixed, the color functions do not depend on time. However, they allow to represent the state variables of the coupled system as the sum of prolongations of the variables of each subsystem to the composed domain $[a,b]$. By abuse of notation, for all $t \geq 0$ and $x(\cdot,t) := (x_1(\cdot,t),x_2(\cdot,t))^{\top} \in X= L^2([a,b], \mathbb{R}^2)$, we have
\begin{align}
\label{aggregate state x}
	x(z,t) &= x^-(z,t) + x^+(z,t) = c_l^-(z)x(z,t) + c_l^+(z)x(z,t).
\end{align}

Furthermore, the flux variable $\mathcal{N}(x) \in X$ of the coupled system becomes
\begin{align*}
	\mathcal{N}(x) := c_l^- \mathcal{N}^-(x) + c_l^+ \mathcal{N}^+(x), \quad x \in X, 
\end{align*}
which obviously satisfies 
\begin{align*}
	c_l^- \mathcal{N}(x) =  c_l^- \mathcal{N}^-(x) \quad \text{and} \quad	c_l^+ \mathcal{N}(x)  =  c_l^+\mathcal{N}^+(x).
\end{align*} 
For the sake of brevity, for $i = 1,2,$ we will sometimes write $\mathcal{N}_i^{\pm}(z,t)$ instead of $\mathcal{N}_i^{\pm}(x)(z,t)$, and $\mathcal{N}_i(z,t)$ instead of $\mathcal{N}_i(x)(z,t) = \mathcal{N}_i(x(\cdot,t))(z)$, where $\mathcal{N}_i^{\pm}(z,t)$ and $\mathcal{N}_i(z,t)$ denote the $i$-th entry of the two-dimensional flux variables $\mathcal{N}^{\pm}(z,t)$ and $\mathcal{N}(z,t)$, respectively.

\subsection{Derivation of the Balance Equations}
\label{Subsection Derivation of the Balance Equations}

To define the coupled system as a port-Hamiltonian system on the state space $X$, we derive next the balance equations of the variables $x_1$ and $x_2$ having regard to the interface relations \eqref{Interface Variable fI}-\eqref{Interface Variable eI}. 

We will show this only for the state variable $x_1$. For further details and for the balance equation regarding $x_2$, we refer to  \cite{Diagne2013}, \cite[Section 5.2]{Kilian2022}. Moreover, as we will show, some equations are satisfied only in a distributional sense or when considering weak derivatives; please consult \cite[Sections 7.2 and 7.5]{Aubin2000} for further reading.

Assume that the individual conservation laws \eqref{Systems of Conservation Laws} are satisfied in the weak sense. That is, there are functions $x^- \in H^1(\Omega^-, \mathbb{R}^2$), $x^+ \in H^1(\Omega^+, \mathbb{R}^2$), where $\Omega^- = (a,l) \times (0,\infty)$ and $\Omega^+ = (l,b) \times (0, \infty)$, so that for all test functions $\varphi^{\pm} \in \mathcal{D}(\Omega^{\pm})$ it holds
\begin{equation}
\label{eq:Weak-solutions-subsystems}
    \int_{\Omega^{\pm}} \left( \frac{\partial}{\partial t} x^{\pm} + \frac{\partial}{\partial z} \mathcal{N}^{\pm}(x^{\pm}) \right) \varphi^{\pm} = 0.
\end{equation}
For these functions $x^{\pm}$, we define the function $x$ as in \eqref{aggregate state x}. Additionally, let us assume that the interface condition \eqref{Interface Variable fI} is satisfied as well. Then, the conservation law for $x_1$ of the aggregate system may be formally written as
\begin{align}
	\begin{split}
	\frac{\partial}{\partial t} x_1 &= - \frac{\partial}{\partial z} \mathcal{N}_1(x) \\
	&= - \frac{\partial}{\partial z} \big( c_l^- \mathcal{N}_1^-(x) + c_l^+\mathcal{N}_1^+(x)  \big) \\
	&= - \frac{\partial}{\partial z} \left( c_l^- \mathcal{N}_1(x) + c_l^+ \mathcal{N}_1(x) \right) \\
	&= \underbrace{- \left[ \frac{\partial}{\partial z} ( c_l^- \cdot) + \frac{\partial}{\partial z} (c_l^+\cdot) \right]}_{=: \mathbf{d}_l} \mathcal{N}_1(x).
\end{split}
\label{Conservation Law of x1}
\end{align}
Note that, in \eqref{Conservation Law of x1},
the difference between $\frac{\partial}{\partial z}(c_l^{\pm} \cdot)$ and the notation $\frac{\partial}{\partial z}c_l^{\pm}$ is the following:
\begin{align*}
    \frac{\partial}{\partial z}(c_l^{\pm} \cdot) \colon \mathcal{N}_1(x) \mapsto \frac{\partial}{\partial z}(c_l^{\pm} \mathcal{N}_1(x)), \qquad
    \frac{\partial}{\partial z}c_l^{\pm} \colon \mathcal{N}_1(x) \mapsto \left( \frac{\partial}{\partial z}c_l^{\pm}\right) \mathcal{N}_1(x),
\end{align*}
i.e., in the former, the expression $c_l^{\pm} \mathcal{N}_1(x)$ is differentiated (in a distributional sense), and in the latter, the function $\mathcal{N}_1(x)$ is multiplied by the (distributional) derivative of $c_l^{\pm}$, yielding a point evaluation of $\mathcal{N}_1^{\pm}(x)$ at the interface position $l^{\pm}$. In the following, we will fix $t > 0$ and justify that the formal calculations in equation \eqref{Conservation Law of x1} make sense at least in the distributional sense. Note also that, due to the formal equality of the first and the last line in \eqref{Conservation Law of x1}, we expect that the maximal domain of the operator $\mathbf{d}_l$, $D_{\max}(\mathbf{d}_l)$, satisfies  $H^1([a,b], \mathbb{R}) \subset D_{\max}(\mathbf{d}_l).$
%Next we are going to justify the above computations.
\begin{proposition}
\label{prop:Balance-equation-distributional-sense}
Let $x$ be defined as in \eqref{aggregate state x}, where $x^{\pm}$ are weak solutions of the systems of conservation laws \eqref{Systems of Conservation Laws}, satisfying the continuity condition \eqref{Interface Variable fI}. Then equation (\ref{Conservation Law of x1}) holds in a distributional sense for all $t >0$, that is, for every test function $\varphi \in \mathcal{D}(a,b)$ it holds that
\begin{align*}
     \int_{a}^{b} \left[ \mathbf{d}_l  \mathcal{N}_1(x) (z,t) \right] \varphi(z) \, dz = \int_{a}^{b} \frac{\partial}{\partial t} x_1(z,t) \varphi(z) \, dz.
\end{align*}
\end{proposition}

\begin{proof}
Let $\varphi \in \mathcal{D}(a,b)$ be some test function. By virtue of \eqref{Systems of Conservation Laws} and due to the the continuity condition \eqref{Interface Variable fI} of $\mathcal{N}_1(x)$, it is easy to check that $\mathcal{N}_1(x(\cdot,t)) \in H^1([a,b], \mathbb{R})$.

Thus, the following holds for all $t > 0$:
\allowdisplaybreaks
\begin{align*}
		\allowdisplaybreaks
 \int_{a}^{b} \left[ \mathbf{d}_l  \mathcal{N}_1(x) (z,t) \right] \varphi(z) dz
	&= \int_{a}^{b} - \left[ \frac{\partial}{\partial z} \Big(c_l^-(z) \mathcal{N}_1(z,t)\Big) + \frac{\partial}{\partial z} \Big(c_l^+ (z) \mathcal{N}_1(z,t)\Big) \right] \varphi(z) \,  dz  \\
\text{(part. int.)} \quad &= \int_{a}^{b} c_l^-(z) \mathcal{N}_1(z,t) \frac{d}{dz} \varphi(z) \, dz + \int_{a}^{b} c_l^+(z) \mathcal{N}_1(z,t) \frac{d}{dz} \varphi(z) \, dz \\
	&= \int_{a}^{l} \mathcal{N}_1^-(z,t) \frac{d}{dz} \varphi(z) \, dz + \int_{l}^{b}  \mathcal{N}_1^+(z,t) \frac{d}{dz} \varphi(z) \, dz \\
\text{(part. int.)} \quad	&=  \mathcal{N}_1^-(l^-,t) \varphi(l) + \int_{a}^{l}  - \frac{\partial}{\partial z} \mathcal{N}_1^-(z,t)  \varphi(z) \,  dz \\
	&\quad - \mathcal{N}_1^+(l^+,t) \varphi(l)   + \int_{l}^{b} - \frac{\partial}{\partial z} \mathcal{N}_1^+(z,t)  \varphi(z) \, dz \\
\quad	&\stackrel{\eqref{Systems of Conservation Laws}}{=}{} \mathcal{N}_1^-(l^-,t) \varphi(l) + \int_{a}^{l}  \frac{\partial}{\partial t} x_1^-(z,t)  \varphi(z) \,  dz \\
	&\quad - \mathcal{N}_1^+(l^+,t) \varphi(l)   + \int_{l}^{b} \frac{\partial}{\partial t} x_1^+(z,t)  \varphi(z) \, dz \\
 \quad	&\stackrel{\eqref{Interface Variable fI}}{=}{} \int_{a}^{b} \frac{\partial}{\partial t} x_1(z,t) \varphi(z) \, dz. 
\end{align*}
This proves the claim.  \qed
\end{proof}

Note that for all $t >0$, $\varphi \in \mathcal{D}(a,b)$, %\mir{and which $x$?} 
one computes
\begin{align*}
- \int_{a}^{b} \frac{d}{dz} c_l^-(z) \mathcal{N}_1(z,t) \varphi(z) \, dz =  \int_{a}^{l} \frac{\partial}{\partial z} \big(\mathcal{N}_1^-(z,t) \varphi(z)\big) \, dz = \mathcal{N}_1^-(l^-,t) \varphi(l), 
\end{align*}
and
\begin{align*}
	- \int_{a}^{b} \frac{d}{dz} c_l^+(z) \mathcal{N}_1(z,t) \varphi(z) \, dz = \int_{l}^{b} \frac{\partial}{\partial z} \big(\mathcal{N}_1^+(z,t) \varphi(z)\big) \, dz =  - \mathcal{N}_1^+(l^+,t) \varphi(l). 
\end{align*}
This implies that, due to the continuity condition \eqref{Interface Variable fI}, the evaluation of the flux variables of the respective subdomain at the interface, $\left( \frac{\partial}{\partial z}c_l^{\pm}\right) \mathcal{N}_1(x)$, cancel each other out, and equation \eqref{Conservation Law of x1} becomes (when understood in a distributional sense)
\begin{align}
	\begin{split}
	\frac{\partial}{\partial t} x_1(\cdot,t) &= \mathbf{d}_l \mathcal{N}_1(x(\cdot,t))  \\
	&= - c_l^-(\cdot) \frac{\partial}{\partial z} \mathcal{N}_1^-(x(\cdot,t)) - \frac{d}{dz} c_l^-(\cdot)\mathcal{N}_1(x(\cdot,t)) \\
	&\quad -  c_l^+(\cdot) \frac{\partial}{\partial z} \mathcal{N}_1^+(x(\cdot,t)) - \frac{d}{dz} c_l^+(\cdot) \mathcal{N}_1(x(\cdot,t)) \\
	&= - c_l^-(\cdot) \frac{\partial}{\partial z} \mathcal{N}_1^-(x(\cdot,t)) - c_l^+(\cdot) \frac{\partial}{\partial z} \mathcal{N}_1^+(x(\cdot,t)).
\end{split}
	\label{Conservation Law of x1 - Version 2}
\end{align}

Since the imposed regularity conditions imply $\mathcal{N}_1(x(\cdot,t)) \in H^1([a,b], \mathbb{R})$, we have for all $t > 0$,
\begin{align*}
    \mathbf{d}_l \mathcal{N}_1(x(\cdot,t)) = - c_l^-(\cdot) \frac{\partial}{\partial z} \mathcal{N}_1^-(x(\cdot,t)) - c_l^+(\cdot) \frac{\partial}{\partial z} \mathcal{N}_1^+(x(\cdot,t)) \in L^2([a,b], \mathbb{R}).
\end{align*} 
As a consequence, we have the following result:
\begin{proposition}
\label{prop:Balance-equation-a.e.-sense}
Under the assumptions of Proposition \ref{prop:Balance-equation-distributional-sense}, equation \eqref{Conservation Law of x1 - Version 2} holds for all $t > 0$ in the a.e. sense.
\end{proposition}

This motivates the definition of the operator  $\mathbf{d}_l \colon D(\mathbf{d}_l) \subset L^2([a,b], \mathbb{R}) \to L^2([a,b], \mathbb{R})$ in the following way:
\begin{align}
	 D(\mathbf{d}_l) :=& \left\{ x \in L^2([a,b], \mathbb{R}) \mid x_{|(a,l)} \in H^1((a,l), \mathbb{R}), \, x_{|(l,b)} \in H^1((l,b), \mathbb{R}), \, x \in \mathcal{C}([a,b], \mathbb{R}) \right\} \nonumber  \\
	  =& H^1([a,b], \mathbb{R}),
  \label{Operator dl} \\
	\mathbf{d}_l x :=& - \left[ \frac{d}{dz} (c_l^- x) + \frac{d}{dz} (c_l^+ x) \right], \quad x \in D(\mathbf{d}_l) \nonumber.
\end{align}
For a function $x = c_l^- x^- + c_l^+ x^+ \in D(\mathbf{d}_l)$ we have in particular that
\begin{align*}
\mathbf{d}_l x = - c_l^- \frac{d}{dz} x^- - c_l^+ \frac{d}{dz} x^+. 
\end{align*} 
This way, the continuity equation \eqref{Interface Variable fI} of the flux variable $\mathcal{N}_1(x)$ is incorporated into the domain of the operator $\mathbf{d}_l$. 

Taking the balance equation \eqref{Interface Variable eI} into consideration, the conservation law of the state variable $x_2$ can be defined with respect to an operator closely related to $\mathbf{d}_l$. Once again, we refer to \cite{Diagne2013}, \cite[Section 5.2]{Kilian2022} for a detailed discussion. 
This admits a port-Hamiltonian formulation of these conservation laws, as we will see in the following section.

\subsection{Port-Hamiltonian Formulation}
\label{Subsection Port-Hamiltonian Formulation}
	
	Let $X = L^2([a,b], \mathbb{R}^2)$ be the state space and $l \in (a,b)$ be the position of the interface. Consider the system
	\begin{equation}
		\label{Simplified PH-System}
		\frac{\partial}{\partial t} x(z,t) = \mathcal{J}_l \big( \mathcal{Q}_l(z) x(z,t) \big), \quad z \in [a,b] \setminus \{l\}, \quad t >0, 
	\end{equation}
	where $\mathcal{J}_l$ is defined below, and $\mathcal{Q}_l = c_l^- \mathcal{Q}^- + c_l^+ \mathcal{Q}^+ \in \mathcal{L}(X)$ is a coercive matrix multiplication operator. 
 The functions $c_l^-$ and $c_l^+$ are the color functions defined in \eqref{Characteristic Functions c- and c+}.

	We endow $X$ with the inner product $\langle \cdot, \cdot \rangle_{\mathcal{Q}_l} = \frac{1}{2} \langle \cdot, \mathcal{Q}_l \cdot \rangle_{L^2}$, i.e., 
	\begin{equation}
		\langle x , y \rangle_{\mathcal{Q}_l} := \frac{1}{2} \int_{a}^{b} y^{\top}(z) \mathcal{Q}_l(z) x(z) \, dz, \quad x, y \in X,
		\label{Inner Product wrt Q0}
	\end{equation}
	with associated norm $\| \cdot \|_{\mathcal{Q}_l}$.
	Note that $\langle \cdot , \cdot \rangle_{\mathcal{Q}_l}$ is equivalent to the standard inner product $\langle \cdot , \cdot \rangle_{L^2}$ on $X$. 

	The Hamiltonian $H \colon X \to \mathbb{R}$ on the energy space $(X, \| \cdot \|_{\mathcal{Q}_l})$ is defined as
	\begin{equation}
		H(x) := \frac{1}{2} \int_{a}^{b} x^{\top}(z) \mathcal{Q}_l(z) x(z) \, dz =  \|x\|_{\mathcal{Q}_l}^2, \quad x \in X.
		\label{Simplified PH-System - Hamiltonian}
	\end{equation}
	The operator $\mathcal{J}_l \colon D(\mathcal{J}_l) \subset X \to X$ is given by
	\begin{align}
		\begin{split}
			D(\mathcal{J}_l) :=& \left\{ x = (x_1,x_2) \in X \mid x_1 \in D(\mathbf{d}_l^{\ast}), \, x_2 \in D(\mathbf{d}_l)\right\}, \\
			\mathcal{J}_l x :=& \begin{bmatrix}
				0 & \mathbf{d}_l  \\
				- \mathbf{d}_l^{\ast} & 0 
			\end{bmatrix} \begin{bmatrix}
				x_1 \\
				x_2
			\end{bmatrix}
			=
			 \begin{bmatrix}
				\mathbf{d}_l x_2 \\
				- \mathbf{d}_l^{\ast} x_1
			\end{bmatrix}, \quad x \in D(\mathcal{J}_l),
		\end{split}
		\label{Operator Jl}
	\end{align}
	where the operator $\mathbf{d}_l$ has been defined in \eqref{Operator dl}, and its uniquely defined formal adjoint\footnote{For a definition of formally adjoint operators, see \cite[Chapter 13, p. 309]{Aubin2000}.} $\mathbf{d}_l^{\ast}$ is given by
	\begin{align*}
		D(\mathbf{d}_l^{\ast}) :=& \left\{ y \in L^2([a,b], \mathbb{R}) \mid y_{|(a,l)} \in H^1((a,l), \mathbb{R}), \, y_{|(l,b)} \in H^1((l,b), \mathbb{R}) \right\}, \\
		\mathbf{d}_l^{\ast}y :=& \left( - \mathbf{d}_l - \left[ \frac{d}{dz} c_l^- + \frac{d}{dz} c_l^+ \right] \right) y, \quad y \in D(\mathbf{d}_l^{\ast}).
	\end{align*}

    \begin{remark}
    \label{Remark Jump Discontinuity}
    In contrast to the domain of $\mathbf{d}_l$, we do not require continuity of the elements in the domain of the formally adjoint operator $\mathbf{d}_l^{\ast}$. Thus, the purely distributional terms (point evaluations at the interface) do not cancel each other out, as it was the case in \eqref{Conservation Law of x1 - Version 2}. Putting the term $- \left[ \frac{d}{dz}c_l^- y + \frac{d}{dz} c_l^+ y \right]$ into the the definition of the operator $\mathbf{d}_l^{\ast}$ to make up for said discontinuity was not an arbitrary choice, though, as elaborated in \cite{Diagne2013}, \cite[Section 5.2]{Kilian2022}.
\end{remark}
Similarly to the operator $\mathbf{d}_l$, for a function $y = c_l^- y^- + c_l^+ y^+ \in D(\mathbf{d}_l^{\ast})$ it holds that
	\begin{equation*}
		\mathbf{d}_l^{\ast} y = c_l^- \frac{d}{dz} y^- + c_l^+ \frac{d}{dz} y^+.
	\end{equation*} 
	For all $x \in D(\mathbf{d}_l), \,  y \in D(\mathbf{d}_l^{\ast})$, we have the following relation between the operators $\mathbf{d}_l$ and $\mathbf{d}_l^{\ast}$:
	\begin{equation}
		\label{Relation dl and dl_ast}
		\langle \mathbf{d}_l x , y \rangle_{L^2} =  - \big[ x(z) y(z) \big]_{a}^{b} + x(l) \left[ y(l^+) - y(l^-) \right] + \langle x , \mathbf{d}_l^{\ast} y \rangle_{L^2}.
	\end{equation} 
    In particular, equation \eqref{Relation dl and dl_ast} describes the formal adjointness of $\mathbf{d}_l$ and $\mathbf{d}_l^{\ast}$, that is, the adjointness up to boundary and interface evaluations. As we will see, the first term on the right-hand side of \eqref{Relation dl and dl_ast} allows for a power flow at the boundary of the system. In particular, the energy is, in general, not conserved, and we allow the system to interact with its environment through the spatial boundary, as is the case for classical boundary port-Hamiltonian systems. The second term on the right-hand side of \eqref{Relation dl and dl_ast} stems from the interface relations \eqref{Interface Variable fI}-\eqref{Interface Variable eI} that are modeled in the definition of the operators $\mathbf{d}_l$ and $\mathbf{d}_l^{\ast}$, respectively, and corresponds to a potential power flow at the interface position. Recalling the examples for motivating the consideration of \eqref{Interface Variable fI}-\eqref{Interface Variable eI}, $x(l)$ represents the continuity at the interface of one physical quantity (current/gas velocity), while due to a jump discontinuity, the evaluations of the other physical quantity at both sides of the interface, $ y(l^+) - y(l^-)$, should not cancel each other (voltage drop across the resistor/pressure jump at the position of the piston), see Remark~\ref{Remark Jump Discontinuity}. 

	Note that on any subinterval of $[a,b]$ not containing the interface position $l \in (a,b)$, the operator $\mathcal{J}_l$ simply acts as the matrix differential operator $P_1 \frac{d}{dz}$, where $P_1$ is defined in \eqref{Matrix P1}. Thus, \eqref{Simplified PH-System} is a reformulation of \eqref{Systems of Conservation Laws}. Next, we reveal the port-Hamiltonian structure of \eqref{Simplified PH-System}.
	
	For $n\in\N$, we introduce the matrix $\Sigma_{2n}$, 	which will be used extensively in the following sections:
	\begin{equation}
		\label{Matrix Sigma}
		\Sigma_{2n} := \begin{bmatrix}
			0 & I_n \\
			I_n & 0
		\end{bmatrix} \in \mathbb{R}^{2n \times 2n}.
	\end{equation}

	In the framework of port-Hamiltonian systems, it is crucial that the associated differential operator is formally skew-symmetric. It describes the power-conserving interconnection of the system components, see, e.g., \cite{Gorrec2005}. As a first step, we show that the operator  $\mathcal{J}_l$ defined in \eqref{Operator Jl} is \emph{formally skew-symmetric}, i.e.,
	\begin{equation*}
	\langle \mathcal{J}_l x, y \rangle_{L^2} = - \langle x, \mathcal{J}_l y \rangle_{L^2}
	\end{equation*}
	for all $x, y \in D(\mathcal{J}_l)$ vanishing both at the boundary and at the interface position.
	\begin{lemma}
		The operator $\mathcal{J}_l \colon D(\mathcal{J}_l) \subset X \to X$ defined in \eqref{Operator Jl} is formally skew-symmetric.
	\end{lemma}

	\begin{proof}
		Applying the relation \eqref{Relation dl and dl_ast} between the operators $\mathbf{d}_l$ and $\mathbf{d}_l^{\ast}$, for $x = (x_1, x_2)$, $y = (y_1, y_2) \in D(\mathcal{J}_l)$ we have
		\begin{align}
			\begin{split}
				&\quad \langle \mathcal{J}_l x , y \rangle_{L^2} + \langle x , \mathcal{J}_l y \rangle_{L^2} \\
				&= \langle \mathbf{d}_l x_2, y_1 \rangle_{L^2} - \langle \mathbf{d}_l^{\ast} x_1, y_2  \rangle_{L^2} + \langle x_1, \mathbf{d}_l y_2 \rangle_{L^2} - \langle x_2, \mathbf{d}_l^{\ast} y_1 \rangle_{L^2} \\
				&= - \big[ y_1(z) x_2(z) \big]_{a}^{b} + x_2(l) \left[ y_1(l^+) - y_1(l^-) \right] \\
				&\quad - \big[ x_1(z) y_2(z) \big]_{a}^{b} + y_2(l) \left[ x_1(l^+) - x_1(l^-) \right] \\
				&= \big[ x^{\top}(z) P_1 y(z) \big]_{a}^{b} + x_2(l) \left[ y_1(l^+) - y_1(l^-) \right]  + y_2(l) \left[ x_1(l^+) - x_1(l^-) \right],
			\end{split} 
			\label{Skew-Symmetry of J0}
		\end{align}
		with $P_1$ defined in \eqref{Matrix P1}. This proves the claim.
        \qed
	\end{proof}
	
	Next, we augment the system \eqref{Simplified PH-System} with boundary and interface port variables. Port variables are used to express the previously mentioned power-conserving interconnection between the system components (\cite[Chapter 2]{Schaft2014}, \cite[Chapter 2]{Villegas2007}). Ports consist of physical quantities whose (inner) product yields power. As the operator $\mathcal{J}_l$ is skew-symmetric up to boundary and interface evaluations, see \eqref{Skew-Symmetry of J0}, we allow for a power flow both at the spatial boundary (hence, an interaction with the system's physical environment), and at the interface position. Thus, by introducing boundary and interface port variables, we can express the power flow of the system at these positions. 
	
	Since the interface position lies in the interior of the spatial domain $[a,b]$, we may define the trace operator $\trace_l \colon D(\mathcal{J}_l) \to \mathbb{R}^4$ as follows:
	\begin{equation*}
	 \trace_l(x) := \begin{bmatrix}
			x(b) \\
			x(a)
		\end{bmatrix}, \quad x \in D(\mathcal{J}_l).
	\end{equation*}
	For all $x \in X$ such that $ \mathcal{Q}_l x \in D(\mathcal{J}_l)$, the boundary flow $f_{\partial} = f_{\partial, \mathcal{Q}_lx} \in \mathbb{R}^2$ and the boundary effort $e_{\partial} = e_{\partial, \mathcal{Q}_l x} \in \mathbb{R}^2$ are defined as
	\begin{equation}
		\label{Boundary Flow and Effort}
		\begin{bmatrix}
			f_{\partial, \mathcal{Q}_l x}\\
			e_{\partial, \mathcal{Q}_l x}
		\end{bmatrix} :=  R_{\extern} \trace_l(\mathcal{Q}_l x) = \begin{bmatrix}  \frac{1}{\sqrt{2}} (P_1 (\mathcal{Q}_l x)(b) - P_1 (\mathcal{Q}_lx)(a)) \\
			\frac{1}{\sqrt{2}} ((\mathcal{Q}_l x)(b) + (\mathcal{Q}_lx)(a))
		\end{bmatrix}, 
	\end{equation} 
	where
	\begin{equation}
		\label{Matrix Rext}
		R_{\extern} := \frac{1}{\sqrt{2}} \begin{bmatrix}
			P_1 & - P_1 \\
			I_2 & I_2 
		\end{bmatrix} \in \mathbb{R}^{4 \times 4}. 
	\end{equation}
	Now, we define the interface port variables $e_{I} = e_{I, \mathcal{Q}_l x}$, $f_I = f_{I, \mathcal{Q}_l x} \in \mathbb{R}$. For all $x \in X$ such that $\mathcal{Q}_l x \in D(\mathcal{J}_l)$, the interface flow and the interface effort are given as follows:
	\begin{align}
		f_{I, \mathcal{Q}_l x} :=& \left( \mathcal{Q}_l x \right)_2 (l^+) = \left( \mathcal{Q}_l x \right)_2(l^-), \label{Continuity Equation} \\
		-e_{I, \mathcal{Q}_l x} :=& \left( \mathcal{Q}_l x \right)_1(l^+) - \left( \mathcal{Q}_l x \right)_1 (l^-). \label{Balance Equation}
	\end{align}
	For all $ \mathcal{Q}_lx, \mathcal{Q}_ly \in D(\mathcal{J}_l)$, we may write the last line of \eqref{Skew-Symmetry of J0} with respect to the boundary and interface port variables as follows: 
	\begin{align}
		\begin{split}
			&\quad \big[ (\mathcal{Q}_lx)^{\top}(z) P_1 (\mathcal{Q}_ly)(z) \big]_{a}^{b} + \left( \mathcal{Q}_l x \right)_2(l) \left[ \left( \mathcal{Q}_l y \right)_1(l^+) - \left( \mathcal{Q}_l y \right)_1(l^-) \right]   \\
			&\quad + \left( \mathcal{Q}_l y \right)_2(l) \left[ \left( \mathcal{Q}_l x \right)_1(l^+) - \left( \mathcal{Q}_l x \right)_1(l^-) \right] \\
			&= \langle e_{\partial, \mathcal{Q}_l y} , f_{\partial, \mathcal{Q}_l x} \rangle + \langle e_{\partial, \mathcal{Q}_l x} , f_{\partial, \mathcal{Q}_l y} \rangle - f_{I, \mathcal{Q}_l x} e_{I, \mathcal{Q}_l y} - f_{I, \mathcal{Q}_l y} e_{I, \mathcal{Q}_l x}. 
		\end{split}
		\label{Auxiliary Equation Power Pairing}
	\end{align}
	Next, we simply substitute the relation \eqref{Auxiliary Equation Power Pairing} into equation \eqref{Skew-Symmetry of J0}. This yields for all $\mathcal{Q}_l x, \mathcal{Q}_l y \in D(\mathcal{J}_l)$, 
	\begin{align}
		\begin{split}
		&\quad \langle \mathcal{J}_l (\mathcal{Q}_l x), \mathcal{Q}_l y \rangle_{L^2} + \langle \mathcal{Q}_l x, \mathcal{J}_l (\mathcal{Q}_l y) \rangle_{L^2} \\
		&= \langle e_{\partial, \mathcal{Q}_l y} , f_{\partial, \mathcal{Q}_l x} \rangle + \langle e_{\partial, \mathcal{Q}_l x} , f_{\partial, \mathcal{Q}_l y} \rangle - f_{I, \mathcal{Q}_l x} e_{I, \mathcal{Q}_l y} - f_{I, \mathcal{Q}_l y} e_{I, \mathcal{Q}_l x}.
		\end{split}
		\label{Skew-Symmetry of J0 - Representation 2}
	\end{align}

	This allows us to state a (power) balance equation for classical solutions $x \in C^1([0, \infty), X)$ of the first-order system \eqref{Simplified PH-System} by means of the boundary and interface port variables.
	\begin{lemma}
		\label{Lemma Balance Equation}
  Let $x \in C^1([0, \infty), X)$ be a classical solution of the system \eqref{Simplified PH-System} with Hamiltonian \eqref{Simplified PH-System - Hamiltonian}. Then we have for all $t \geq 0$,
	
		\begin{equation}
			\label{Power Balance Equation wrt Boundary and Interface Port}
			\frac{d}{dt} \|x(t)\|_{\mathcal{Q}_l}^2 =	\frac{d}{dt}H(x(t)) =  \langle e_{\partial}(t), f_{\partial}(t) \rangle - e_I(t) f_I(t),
		\end{equation}
	\end{lemma}
where $e_{\partial}(t) = e_{\partial, \mathcal{Q}_lx(t)}$, $f_{\partial}(t) = f_{\partial, \mathcal{Q}_lx(t)}$, $e_I(t) = e_{I, \mathcal{Q}_lx(t)}$, and $f_I(t) = f_{I, \mathcal{Q}_lx(t)}$.

\begin{proof}
		By exploiting the fact that $x$ satisfies \eqref{Simplified PH-System} and by using \eqref{Skew-Symmetry of J0 - Representation 2}, we get for all $t \geq 0$, 
		\begin{align*}
			\frac{d}{dt}H(x(t)) &= \frac{1}{2} \frac{d}{dt} \langle x(t), \mathcal{Q}_lx(t) \rangle_{L^2} \\
			&= \frac{1}{2} \langle \mathcal{J}_l( \mathcal{Q}_lx(t) ), \mathcal{Q}_l x(t) \rangle_{L^2} + \frac{1}{2} \langle \mathcal{Q}_l x(t), \mathcal{J}_l(\mathcal{Q}_lx(t)) \rangle_{L^2} \\
			&= \langle e_{\partial}(t), f_{\partial}(t) \rangle - e_I(t) f_I(t).
		\end{align*}
        \qed
\end{proof}
	Equation \eqref{Power Balance Equation wrt Boundary and Interface Port} states that the change of energy is equal to the power flow both at the boundary and at the interface position.
	
\begin{remark}[Dirac Structure of \eqref{Simplified PH-System}]
Note that the existence of an underlying, power-conserving interconnection structure, called a (Stokes-)Dirac structure, is constitutive for port-Hamiltonian systems \cite[Chapter 2]{Schaft2014}, \cite{Gorrec2005},  \cite[Chapter 2]{Villegas2007}. It allows us to geometrically specify the dynamics of the associated system by means of its constituting port variables. However, in this work, we will not address this topic. The underlying Dirac structure of the system \eqref{Simplified PH-System} has been defined in \cite[Subsection 5.4.1]{Kilian2022}.
\end{remark}
		
	Now that we have established the port-Hamiltonian formulation of the system \eqref{Simplified PH-System}, we want to discuss under which conditions this system is well-posed. In the subsequent sections, we will assume (without loss of generality) that the fixed interface position is given by $l = 0$.

	\section{Generation of a Contraction Semigroup}
	\label{Section Generation of a Contraction Semigroup}

	In this section, we characterize the boundary and interface conditions under which the system \eqref{Simplified PH-System} generates a contraction semigroup. In particular, this
    guarantees well-posedness in the sense of \cite[Definition II.6.8]{Engel2000}.
	 To this end, we investigate the port-Hamiltonian operator associated with the system \eqref{Simplified PH-System} with Hamiltonian \eqref{Simplified PH-System - Hamiltonian}. In \cite[Section 4]{Gorrec2005}, it has been shown that 
   boundary port variables can be used to characterize those boundary conditions for which the associated port-Hamiltonian operator generates a contraction semigroup. With this approach in mind, we will introduce interface and boundary conditions specified by the interface, respectively, boundary port variables. These conditions will be incorporated into the definition of the domain of the port-Hamiltonian operator. 
  It will then turn out that the problem admits a reformulation as a pure boundary port-Hamiltonian system of the type studied in \cite{Jacob2012}. Using this reformulation, we show in Theorem~\ref{Theorem A Generates a Contraction Semigroup} that the port-Hamiltonian operator from \eqref{Simplified PH-System} generates a contraction semigroup if and only if certain  passivity conditions both at the spatial boundary and at the interface position hold. Furthermore, criteria are presented that guarantee the generation of an isometric semigroup by the port-Hamiltonian operator.

\subsection{Boundary and Interface Conditions}

	We begin by introducing boundary conditions for the system \eqref{Simplified PH-System}. Following \cite[Section 4]{Gorrec2005}, we impose linear boundary conditions of the form 
	\begin{equation}
		\tilde{W}_B \begin{bmatrix}
			(\mathcal{Q}_0x)(b) \\
			(\mathcal{Q}_0x)(a) 
		\end{bmatrix} = 0
		\label{Boundary Condition}
	\end{equation}
	for some $\tilde{W}_B \in \mathbb{R}^{2 \times 4}$. Since the matrix  $R_{\extern}$ defined in \eqref{Matrix Rext} is invertible, we may write \eqref{Boundary Condition} equivalently in terms of the boundary port variables \eqref{Boundary Flow and Effort} as
	\begin{equation}
		W_B \begin{bmatrix}
			f_{\partial, \mathcal{Q}_0x} \\
			e_{\partial, \mathcal{Q}_0 x}
		\end{bmatrix} = 0, 
		\qquad W_B := \tilde{W}_B R_{\extern}^{-1} \in \mathbb{R}^{2 \times 4}.
		\label{Boundary Condition wrt Boundary Flow and Effort}
	\end{equation}
	
	Next, we impose an \emph{interface relation} upon the interface port variables \eqref{Continuity Equation}-\eqref{Balance Equation}. For a fixed parameter $r \in \mathbb{R}$, it is given by
	\begin{equation}
		\label{Interface Passivity Relation}
		f_{I, \mathcal{Q}_0 x} = r e_{I, \mathcal{Q}_0 x}.
	\end{equation}
	Note that due to the minus sign on the right-hand side of the balance equation \eqref{Power Balance Equation wrt Boundary and Interface Port}, the interface relation \eqref{Interface Passivity Relation} is in fact a passivity relation with passivity constant $r$ (provided $r \geq 0$), see \cite[Section 2.4]{Schaft2014}. Recall that the interface port variables are defined by
	\begin{align}
	\label{eq:Interface-variables-at-0}
		f_{I, \mathcal{Q}_0 x} &= \left( \mathcal{Q}^+x\right)_2 (0^+) = \left( \mathcal{Q}^- x \right)_2 (0^-), \\
		-e_{I, \mathcal{Q}_0 x} &= \left( \mathcal{Q}^+x \right)_1 (0^+) - \left( \mathcal{Q}^- x \right)_1 (0^-). \nonumber
    \end{align}
	With these boundary and interface conditions, we denote by  $A_{\mathcal{Q}_0} \colon D(A_{\mathcal{Q}_0}) \subset X \to X$ the port-Hamiltonian operator associated with the system \eqref{Simplified PH-System} with Hamiltonian \eqref{Simplified PH-System - Hamiltonian}. It is given by
	\begin{align}
		\begin{split}
			D(A_{\mathcal{Q}_0}) :=& \left\{ x \in X \, \big| \, \mathcal{Q}_0x \in D(\mathcal{J}_0), \, f_{I, \mathcal{Q}_0x} = r e_{I, \mathcal{Q}_0x}, \, W_B \begin{bmatrix}
				f_{\partial, \mathcal{Q}_0 x} \\
				e_{\partial, \mathcal{Q}_0 x}
			\end{bmatrix} = 0 \right\}, \\
			A_{\mathcal{Q}_0}x :=& \mathcal{J}_0 (\mathcal{Q}_0x), \quad  x \in D(A_{\mathcal{Q}_0}).
		\end{split}
		\label{Simplified PH System - Operator A}
	\end{align}
	We wish to characterize the boundary and interface conditions (i.e., those $W_B$ and $r$) for which $A_{\mathcal{Q}_0}$ generates a contraction semigroup $(T(t))_{t \geq 0}$ on the energy space $(X, \| \cdot \|_{\mathcal{Q}_0})$. 
	In particular, such conditions guarantee the well-posedness of the abstract Cauchy problem
		\begin{align}
			\begin{split}
				\dot{x}(t) &= A_{\mathcal{Q}_0}x(t), \quad t > 0, \\
				x(0) &= x_0 \in D(A_{\mathcal{Q}_0}),
			\end{split}
			\label{Theorem A Generates a Contraction Semigroup - ACP}
		\end{align}
		associated with the operator $A_{\mathcal{Q}_0}$.
	
	\subsection{A Reformulation as a Boundary Port-Hamiltonian System}
	
	To characterize well-posedness of our port-Hamiltonian system with interface, we reformulate it in terms of a system without interface in the framework of \cite{Jacob2012}. This allows to use the well-posedness criteria for port-Hamiltonian systems developed in \cite{Jacob2012} for the analysis of our model.  We continue to assume $l=0$.
	
	Let $\eta := -a/b > 0$ and consider the transformation $w_1(\zeta,t) := x^-(\eta \zeta +a,t)$, $\zeta\in[0,b]$. With the transformation $z = \eta \zeta +a$, the differential equation for $x^-$ on $[a,0]$ given by
	\begin{equation}
	    \frac{\partial}{\partial t} x^-(z,t) = P_1  \frac{\partial}{\partial z} \left(
	    \mathcal{Q}^-(z) x^-(z,t) \right), \quad z\in (a,0), \quad t>0,
	\end{equation}
	transforms to the equation on $(0,b)$ given by
	\begin{equation}
	    \frac{\partial}{\partial t} w_1(\zeta,t) = P_1  \frac{\partial}{\partial \zeta} \left( \frac{1}{\eta}
	    \mathcal{Q}^-(\eta\zeta +a) w_1(\zeta,t) \right), \quad \zeta\in (0,b), \quad t>0.
	\end{equation}
	
	We now introduce a system of four first-order PDEs on $[0,b]$ that consists of copies of the original system on $[a,0]$ (in a transformed form) and $[0,b]$. The system is given by
	\begin{equation}
	    \frac{\partial}{\partial t} w(\zeta,t) = \hat{P}_1  \frac{\partial}{\partial \zeta} \left(
      \mathcal{H}_0w\right)(\zeta,t), \quad t>0,\quad \zeta \in (0,b),
	\end{equation}
	where $\hat{P}_1 := \dg (P_1,P_1)\in \R^{4\times 4}$, and
	\[
	\mathcal{H}_0(\zeta):= \dg (\eta^{-1}\mathcal{Q^-}(\eta \zeta + a),\mathcal{Q}^+(\zeta)),\quad \zeta\in[0,b].
	\]
	Let $\hat{X}:= L^2([0,b],\R^4)$ be endowed with the scalar product $\langle \cdot,\cdot \rangle_{\mathcal{H}_0}$ (defined analogously to \eqref{Inner Product wrt Q0}) and consider the map 

\begin{subequations}
 \begin{align}
        \label{eq:T-map-1}
	   S \colon \hat{X}\to X, \quad w &= \begin{bmatrix}
	        w_1 \\ w_2
	    \end{bmatrix} \mapsto x:= \frac{1}\eta{}c^- w_1(\eta^{-1}\cdot+b)
	    + c^+ w_2,\\
     \label{eq:T-map-2}
	    S^{-1} \colon X \to \hat{X}, \quad x &= c^- x^-  
	    + c^+ x^+ \mapsto
	    w := \begin{bmatrix}
	        \eta x^-(\eta \cdot +a) \\ x^+
	    \end{bmatrix}. 
	\end{align}
\end{subequations}
	The relation between $X$ and $\hat{X}$ is established via the following lemma:
	\begin{lemma}
	\label{lem:Isometric-isomorphism}
	$S$ induces an isometric isomorphism of the Hilbert spaces $(X,\|\cdot\|_{\mathcal{Q}_0})$ and $(\hat{X},\|\cdot\|_{\mathcal{H}_0})$.    
	\end{lemma}

	\begin{proof}
 For $x,y\in X$, denote $w:=S^{-1}y, v:=S^{-1}x$, where $S^{-1}$ is given by \eqref{eq:T-map-2}. We have
	\begin{align*}
		2 \langle x , y \rangle_{\mathcal{Q}_0} &=  \int_{a}^{b} y^{\top}(z) \mathcal{Q}_0(z) x(z) \, dz \\
		&= \int_{a}^{0} y^{-,\top}(z) \mathcal{Q}^-(z) x^-(z) \, dz
		+ \int_{0}^{b} y^{+,\top}(z) \mathcal{Q}^+(z) x^+(z) \, dz.
	\end{align*}		
Substituting $z= \eta \zeta +a$, $ dz=  \eta \, d\zeta$ in the first integral and writing $\zeta:=z$ in the second integral, we see that
	\begin{align*}
		2 &\langle x , y \rangle_{\mathcal{Q}_0}  
		= \int_{0}^{b} y^{-,\top}(\eta \zeta +a)  \mathcal{Q}^-(\eta\zeta +a) x^-(\eta\zeta +a) \eta \, d\zeta
		+ \int_{0}^{b} y^{+,\top}(\zeta) \mathcal{Q}^+(\zeta) x^+(\zeta) \, d\zeta  \\
		&= \int_0^b \begin{bmatrix}
		     y^{-,\top}(\eta \zeta +a) & y^{+,\top}(\zeta)
		\end{bmatrix} \begin{bmatrix}
		    \eta  \mathcal{Q}^-(\eta \zeta +a) & 0  \\
		    0 & \mathcal{Q}^+(\zeta)
		\end{bmatrix} \begin{bmatrix}
		    x^{-}(\eta \zeta +a) \\ x^{+}( \zeta)
		\end{bmatrix} d\zeta \\
	&=	\int_0^b
        \begin{bmatrix}
		     w_1(\zeta) \\ w_2(\zeta)
		\end{bmatrix}^{\top} \begin{bmatrix}
		    \eta^{-1}  \mathcal{Q}^-(\eta \zeta +a) & 0 \\
		    0 & \mathcal{Q}^+(\zeta)
		\end{bmatrix} 
		\begin{bmatrix}
		     v_1(\zeta) \\ v_2(\zeta)
		\end{bmatrix} d\zeta = 2 \left\langle \begin{bmatrix}
		     w_1 \\ w_2
		\end{bmatrix},\begin{bmatrix}
		     v_1 \\ v_2
		\end{bmatrix}\right\rangle_{\mathcal{H}_0}.
	\end{align*}	
    \qed
	\end{proof}
	
	 Recall $r$ from the interface condition \eqref{Interface Passivity Relation} and define the auxiliary matrices
	\begin{equation}
	    U_1 := \begin{bmatrix}
	        - r & 1 \\ 0 & 1
	    \end{bmatrix}, \qquad U_2 := \begin{bmatrix}
	        r & 0 \\ 0 & -1
	    \end{bmatrix}.
	\end{equation}
	With \eqref{eq:Interface-variables-at-0}, it is easy to see that the conditions \eqref{Continuity Equation} and \eqref{Balance Equation} together with the interface condition $f_I = re_I$ are equivalent to
	\begin{equation*}
	    U_1 (\mathcal{Q}^-x^-)(0^-) + U_2 (\mathcal{Q}^+x^+)(0^+) = 0.
	\end{equation*}
	For $\tilde{W}_B \in \R^{2\times 4}$ and $r\in \R$, introduce the boundary condition matrix $\tilde{W}_{B,c}$ by
	\begin{equation}
	    \tilde{W}_{B,c} := \begin{bmatrix}
	        U_1 & 0_{2\times 4} & U_2 \\
	        0_{2\times 2} & \tilde{W}_B & 0_{2\times 2}
	    \end{bmatrix} \in \R^{4 \times 8}.
	\end{equation}
	Note that with this definition and the relation $w = S^{-1}x$, we have that
	\begin{align*}
	   \tilde{W}_{B,c} \begin{bmatrix}
	       (\mathcal{H}_0 w)(b) \\ (\mathcal{H}_0 w)(0)
	   \end{bmatrix} 
	   = \tilde{W}_{B,c} \begin{bmatrix}
	       (\mathcal{H}_0 w)_1(b) \\ (\mathcal{H}_0 w)_2(b) \\ (\mathcal{H}_0 w)_1(0) \\ (\mathcal{H}_0 w)_2(0)
	   \end{bmatrix} 
	   &=  \tilde{W}_{B,c} \begin{bmatrix}
	       (\mathcal{Q}^-x^-)(0^-) \\ (\mathcal{Q}^+x^+)(b) \\ (\mathcal{Q}^-x^-)(a) \\ (\mathcal{Q}^+x^+)(0^+)
	   \end{bmatrix} 	  
	     \\
	   &=\begin{bmatrix}
	     U_1 (\mathcal{Q}^- x^-)(0^-)+ 
	     U_2 (\mathcal{Q}^+ x^+)(0^+)\\
	     \tilde{W}_B \begin{bmatrix}
	         (\mathcal{Q}^+ x^+)(b) \\ 
	         (\mathcal{Q}^- x^-)(a)
	     \end{bmatrix}
	   \end{bmatrix}.
	\end{align*}
	In particular, the boundary condition
	$\tilde{W}_{B,c} \begin{bmatrix}
	    (\mathcal{H}_0 w)(b) \\ (\mathcal{H}_0 w)(0)
	\end{bmatrix}= 0$ is equivalent to the conditions \eqref{Continuity Equation}, \eqref{Balance Equation}, $f_I = re_I$, and $\tilde{W}_B\begin{bmatrix}
	    (\mathcal{Q}^+x^+)(b)\\(\mathcal{Q^-}x^-)(a)
	\end{bmatrix} = 0$ under the transformation $S$.

	The transformation $\hat{R}_{\extern}$ associated to $\hat{P}_1$ is given as before by
	\begin{equation}
	    \hat{R}_{\extern} := \frac{1}{\sqrt{2}} \begin{bmatrix}
	        \hat{P}_1 & - \hat{P}_1 \\
	        I_4 & I_4
	    \end{bmatrix} \in \R^{8\times 8}.
	\end{equation}
	Again, we define $W_{B,c} := \tilde{W}_{B.c} \hat{R}_{\extern}^{-1}$. 
    A direct computation shows that $\hat{R}_{\extern}^{-1} = \hat{R}_{\extern}^{\top}$.	Also, in view of \cite[Lemma 7.2.2]{Jacob2012}, it holds that 
    \begin{align}
    \label{eq:RSigma}
    \hat{R}_{\extern}^{\top}\Sigma_8\hat{R}_{\extern} = \begin{bmatrix}
	        \hat{P}_1 & 0 \\
	        0 & - \hat{P}_1
	    \end{bmatrix}. 
    \end{align}
    The following observation is required for the final result.
	\begin{lemma} 
	\label{lem:WB-WBc}
	The following two statements are equivalent:
	\begin{enumerate}
	    \item[(i)] $W_B \Sigma_4 W_B^{\top}\geq 0$ and $r\geq0$.
	    \item[(ii)] $W_{B,c} \Sigma_8 W_{B,c}^{\top} \geq 0$.
	\end{enumerate}
	\end{lemma}
	\begin{proof}
	Using \eqref{eq:RSigma}, we obtain that 
	\begin{align*}
	    W_{B,c} \Sigma_8 W_{B,c}^{\top} 
	    &= \tilde{W}_{B,c}\hat{R}_{\extern}^{-1}\Sigma_8 (\hat{R}_{\extern}^{-1})^{\top}  \tilde{W}^{\top}_{B,c}
	    = \tilde{W}_{B,c}\hat{R}_{\extern}^{\top}\Sigma_8 \hat{R}_{\extern} \tilde{W}^{\top}_{B,c}\\
	    &= \tilde{W}_{B,c} \dg (P_1,P_1,-P_1,-P_1) \tilde{W}^{\top}_{B,c} \\
            &=
	    \dg (U_1 P_1 U_1^{\top} - U_2 P_1 U_2^{\top}, W_B \Sigma_4 W_B^{\top} ). 
	\end{align*}	
	The assertion thus depends on the definiteness of
	\begin{align*}
     U_1 P_1 U_1^{\top} - U_2 P_1 U_2^{\top} &=
	    \begin{bmatrix}
	        - r & 1 \\ 0 & 1
	    \end{bmatrix} 
	    \begin{bmatrix}
	        0 & -1 \\ -1 & 0
	    \end{bmatrix} \begin{bmatrix}
	        - r & 0 \\ 1 & 1
	    \end{bmatrix} + \begin{bmatrix}
	        r & 0 \\ 0 & -1
	    \end{bmatrix} \begin{bmatrix}
	        0 & 1 \\ 1 & 0
	    \end{bmatrix}\begin{bmatrix}
	        r & 0 \\ 0 & -1
	    \end{bmatrix} \\
	    &=
	    \begin{bmatrix}
	        2r & r\\ r & 0
	    \end{bmatrix} + \begin{bmatrix}
	        0 & -r \\ -r & 0 
	    \end{bmatrix} = \begin{bmatrix}
	        2r & 0 \\ 0 & 0
	    \end{bmatrix}.
	\end{align*}
	The latter matrix is clearly positive semidefinite if and only if $r\geq 0$.
    \qed
	\end{proof}
	
	We will now study the $C_0$-semigroups associated to the operators 
$A_{\mathcal{Q}_0}$ introduced in \eqref{Simplified PH System - Operator A}: 
\begin{align}
		\begin{split}
			D(A_{\mathcal{Q}_0}) &= \left\{ x \in X \, \big| \, \mathcal{Q}_0x \in D(\mathcal{J}_0), \, f_{I, \mathcal{Q}_0x} = r e_{I, \mathcal{Q}_0x}, \, W_B \begin{bmatrix}
				f_{\partial, \mathcal{Q}_0 x} \\
				e_{\partial, \mathcal{Q}_0 x}
			\end{bmatrix} = 0 \right\}, \\
			A_{\mathcal{Q}_0}x &= \mathcal{J}_0 (\mathcal{Q}_0x), \quad  x \in D(A_{\mathcal{Q}_0}),
		\end{split}
		\label{def:A_Q-v2}
	\end{align}
and $A_{\mathcal{H}_0}$ defined next:
    \begin{align}
		\begin{split}
			D(A_{\mathcal{H}_0}) :=& \left\{ x \in \hat{X} \, \big| \, \mathcal{H}_0x \in 
			H^1([0,b], \R^4), \, W_{B .c}\begin{bmatrix}
				f_{\partial, \mathcal{H}_0 x} \\
				e_{\partial, \mathcal{H}_0 x}
			\end{bmatrix} = 0 \right\}, \\
			A_{\mathcal{H}_0}x :=& 
			P_1 \frac{d}{d \zeta}(\mathcal{H}_0x), \quad  x \in D(A_{\mathcal{H}_0}).
		\end{split}
		\label{def:A_H}
	\end{align}

	We note the following simple observation, which follows directly from the previous arguments in this section:
	\begin{lemma}
	    \label{lem:Asimilarity}
	    The operators $A_{\mathcal{Q}_0}$ and  $A_{\mathcal{H}_0}$ are similar via the isometric isomorphism $S$, i.e.,
	    $S D(A_{\mathcal{H}_0}) = D(A_{\mathcal{Q}_0})$
	    and for all $x\in D(A_{\mathcal{H}_0})$ we have
	   $S^{-1} A_{\mathcal{Q}_0} S x = 
	        A_{\mathcal{H}_0} x$.
	    	\end{lemma}
	    	
With this in mind, we can show the following result.	    
\begin{theorem}
		\label{Theorem A Generates a Contraction Semigroup}
		Consider the port-Hamiltonian operator $A_{\mathcal{Q}_0} \colon D(A_{\mathcal
		Q_0}) \subset X \to X$ defined in \eqref{def:A_Q-v2} with a full rank matrix $W_B \in \mathbb{R}^{2 \times 4}$. Then the following assertions are equivalent:
		\begin{enumerate}
		    \item $A_{\mathcal{Q}_0}$ is the infinitesimal generator of a contraction semigroup on $X$ endowed with the scalar product 
		    $\langle \cdot, \cdot \rangle_{\mathcal{Q}_0} = \frac{1}{2} \langle \cdot , \mathcal{Q}_0 \cdot \rangle_{L^2}$.
		    \item $A_{\mathcal{Q}_0}$ is dissipative.
		    \item $r \geq 0$ and with $\Sigma_4$ defined in \eqref{Matrix Sigma}
		    we have
		    \begin{equation}
		        W_B \Sigma_4 W_B^{\top} \geq 0.
		        \label{Matrix Condition W_B}
		    \end{equation}
		\end{enumerate}
		\end{theorem}
\begin{proof}
The assumptions imply that $\hat P_1$ is Hermitian and invertible. Furthermore, we have that $\mathcal{H}_0 \in L^\infty([0,b], \R^{4\times 4})$ is pointwise a.e. symmetric and positive definite, bounded and bounded away from $0$. Also, the matrix $W_{B,c}$ has full rank $4$, as $W_B$ has full rank. Thus, $A_{\mathcal{H}_0}$ meets the assumption of Theorem 7.2.4 in \cite{Jacob2012} and therefore for $A_{\mathcal{H}_0}$ we have the equivalence of the statements: (i) $A_{\mathcal{H}_0}$ generates a contraction semigroup $(T_{\mathcal{H}_0}(t))_{t \geq 0}$, (ii) $A_{\mathcal{H}_0}$ is dissipative, and (iii) $W_{B,c}\Sigma_8 W_{B,c}^{\top} \geq 0$. By Lemma~\ref{lem:Asimilarity}, the operators $A_{\mathcal{Q}_0}$ and $A_{\mathcal{H}_0}$ are similar through an isometric isomorphism. By this similarity, the items 1. and 2. of the assertion and (i), resp. (ii) are equivalent. Finally, by Lemma~\ref{lem:WB-WBc}, item 3. of the claim is equivalent to (iii). This shows the assertion.
\qed
\end{proof}
Condition \eqref{Matrix Condition W_B} is a standard tool for checking passivity of the boundary conditions \eqref{Boundary Condition wrt Boundary Flow and Effort}  (defined through $W_B$ and the boundary port variables $f_{\partial}, e_{\partial}$) of classical boundary port-Hamiltonian systems (see \cite[Section 4]{Gorrec2005}). In that case, energy is conserved if and only if there is no power inflow at the boundary (i.e., no interaction with the environment through the spatial boundary). In our present case, both the passivity criterion \eqref{Matrix Condition W_B} concerning the boundary matrix $W_B$ and the passivity condition $r \geq 0$ concerning the interface condition \eqref{Interface Passivity Relation} ensure that the right-hand side of \eqref{Power Balance Equation wrt Boundary and Interface Port} is bounded from above by zero, i.e., the energy does not increase over time. In fact, Theorem \ref{Theorem A Generates a Contraction Semigroup} characterizes the boundary and interface conditions for which the port-Hamiltonian model of two systems of conservation laws coupled by a fixed interface induces a contraction semigroup.
	The fact that for all $x_0 \in D(A_{\mathcal{Q}_0})$, the abstract Cauchy problem \eqref{Theorem A Generates a Contraction Semigroup - ACP} has a unique classical solution $x \colon [0, \infty) \to X$, depending continuously on the initial data, follows immediately from \cite[Corollary II.6.9]{Engel2000}.

 Recall that a semigroup $(T(t))_{t\ge0}$ over $(X, \|\cdot\|_{\mathcal{Q}_0})$ is called isometric if $\|T(t)x\|_{\mathcal{Q}_0} = \|x\|_{\mathcal{Q}_0}$ for all $t \ge 0$ and $x \in X$.

	  Similarly to Theorem \ref{Theorem A Generates a Contraction Semigroup}, we can characterize the boundary and interface conditions for which the internally stored energy of the system is conserved. Indeed, we see from the proof of Lemma~\ref{lem:WB-WBc} that $W_B\Sigma_4W_B^\top=0$ together with $r=0$ is equivalent to $W_{B,c}\Sigma_8W_{B,c}^\top =0$. We thus obtain the following result from an application of \cite[Theorem 4.1]{Gorrec2005}.  
	\begin{corollary}
		\label{Corollary A Generates an Isometric Semigroup}
		Consider the port-Hamiltonian operator $A_{\mathcal{Q}_0} \colon D(A_{\mathcal{Q}_0}) \subset X \to X$ defined in \eqref{Simplified PH System - Operator A} with $r \in \mathbb{R}$ and with a full rank matrix $W_B \in \mathbb{R}^{2 \times 4}$. Then $A_{\mathcal{Q}_0}$ is the infinitesimal generator of an isometric semigroup on the energy space $(X, \|\cdot\|_{\mathcal{Q}_0})$ if and only if $r =0$ and $W_B$ satisfies $W_B \Sigma_4 W_B^{\top} = 0$.
	\end{corollary}

	\section{Exponential Stability}
	\label{Section Exponential Stability}
	In this section, we present conditions that guarantee exponential stability of the abstract Cauchy problem \eqref{Theorem A Generates a Contraction Semigroup - ACP} associated with the port-Hamiltonian operator $A_{\mathcal{Q}_0}$ defined in \eqref{Simplified PH System - Operator A}. Recall that a $C_0$-semigroup $(T(t))_{t \geq 0}$ on $X$ is called (uniformly) exponentially stable if there exist constants $M \geq 1$ and $\alpha > 0$ such that
		$\|T(t)\|_{\mathcal{L}(X)} \leq M e^{- \alpha t}, t \geq 0$, see, e.g., \cite[Definition 3.11]{Engel2000}.
	Throughout this section, we impose the following assumptions:
	\begin{enumerate}[label = (A\arabic*)]
		\item \label{MatrixQ0Assumption1} The matrix operators $\mathcal{Q}^{\pm}$ defining the coercive operator $\mathcal{Q}_0 \in \mathcal{L}(X)$ satisfy $\mathcal{Q}^{\pm} \in \mathcal{C}^1([a,b], \mathbb{R}^{2 \times 2})$. 
		\item \label{ContractionAssumption} $r \geq 0$, $\rank(W_B) = 2$, and $W_B$ satisfies \eqref{Matrix Condition W_B}.
	\end{enumerate}

 In other words, Assumption \ref{ContractionAssumption} means precisely  that the port-Hamiltonian operator $A_{\mathcal{Q}_0}$ generates a contraction semigroup (see Theorem~\ref{Theorem A Generates a Contraction Semigroup}). Exponential stability of boundary port-Hamiltonian systems is already discussed in \cite[Chapter 9]{Jacob2012}. As we have shown in Lemma \ref{lem:Asimilarity} that our system is similar via an isometric isomorphism to a system of the type discussed in that reference, of course, the results there are immediately applicable. However, the respective conditions translate into estimates on the boundary values of $w$, that is in terms of the original system in conditions on the boundary variables as well as the interface variables. In the context of the systems discussed here, this may be undesirable and we will show that sufficient conditions can be formulated exclusively in terms of evaluations at the boundary points at $a$ and $b$. 
 
	We begin by showing an essential property that helps to verify the exponential stability of the port-Hamiltonian system \eqref{Theorem A Generates a Contraction Semigroup - ACP}. The following lemma is an extension of the energy estimates in \cite[Lemma~9.1.2]{Jacob2012}, which in our setting removes the necessity to evaluate at the interface. In the case $r = 0$ the difference is that, due to the jump discontinuity at the interface position, we need to introduce two auxiliary functions defined on two complementary subintervals to exploit the system dynamics. However, the case $r > 0$ yields an improved energy estimate that only requires the evaluation at one of the boundary points. The result is similar in appearance to  \cite[Lemma~9.1.2]{Jacob2012} for boundary port-Hamiltonian systems without an interface and the proof technique relies in a substantial manner on the tools developed there.

	\begin{lemma}
		\label{Lemma Auxiliary Lemma for Exponential Stability}
		Let the assumptions \ref{MatrixQ0Assumption1}-\ref{ContractionAssumption} hold. Denote by $(T(t))_{t \geq 0}$ the contraction semigroup generated by $A_{\mathcal{Q}_0}$.
		\begin{enumerate}[label = (\roman*)]
		    \item If $r = 0$, then there exist constants $\tau >0$ and $C >0$ such that for every initial value $x_0 \in D(A_{\mathcal{Q}_0})$, the trajectory $x(\cdot) = T(\cdot)x_0$ satisfies
		\begin{equation}
			\label{Lemma Auxiliary Lemma for Exponential Stablity - Inequality}
			\| x(\tau) \|_{\mathcal{Q}_0}^2 \leq C \int_{0}^{\tau} | \mathcal{Q}^-(a)x(a,t) |^2 + | \mathcal{Q}^+(b) x(b,t) |^2 \, dt.
		\end{equation}
		\item If $r > 0$, then there exist constants $\tau >0$ and $C >0$ such that for every initial value $x_0 \in D(A_{\mathcal{Q}_0})$, the trajectory $x(\cdot) = T(\cdot)x_0$ satisfies both
		\begin{equation}
			\label{Lemma Auxiliary Lemma for Exponential Stablity 2 - Inequality 2}
			\begin{split}
			\| x(\tau) \|_{\mathcal{Q}_0}^2 &\leq C \int_{0}^{\tau} | \mathcal{Q}^+(b) x(b,t) |^2 \, dt \quad \text{and} \\
				\| x(\tau) \|_{\mathcal{Q}_0}^2 &\leq C \int_{0}^{\tau} | \mathcal{Q}^-(a)x(a,t) |^2 \, dt.
			\end{split}
		\end{equation} 
		\end{enumerate}
	\end{lemma}
	
	\begin{proof}
	The proof of $(i)$ can be found in \cite[pp. 120-122]{Kilian2022} and we will concentrate on the second assertion. We will prove the first inequality in \eqref{Lemma Auxiliary Lemma for Exponential Stablity 2 - Inequality 2}; the proof of the second being similar.

 Since  \eqref{Lemma Auxiliary Lemma for Exponential Stablity 2 - Inequality 2} is clear for $x_0 = 0$, assume that 
$0 \neq x_0 \in D(A_{\mathcal{Q}_0})$ and let $x(\cdot,\cdot)$ be the corresponding solution. Choose $\gamma > 0$, $\tau > 0$ such that $\tau > 2 \gamma (b-a)$ and define the function $F \colon [a,b]\setminus \{ 0 \} \to [0, \infty)$ by
\begin{equation}
\label{Lemma Auxiliary Lemma for Exponential Stability - Function F}
    F(z) := \int_{\gamma(b-z)}^{\tau - \gamma(b-z)} x^{\top}(z,t) \mathcal{Q}_0(z) x(z,t) \, dt, \quad z \in [a,b]
    \setminus \{ 0 \}. 
\end{equation}
In general, $F$ does not have a continuous extension to  $[a,b]$, as $F(0^-) \neq F(0^+)$. 

Recall that on subdomains not containing the interface position $l = 0$, the operator $\mathcal{J}_0$ defined in \eqref{Operator Jl} acts as the differential operator $P_1 \frac{d}{dz}$, where $P_1$ is given in \eqref{Matrix P1}. Furthermore, by \ref{MatrixQ0Assumption1}, $F$ is differentiable on the subdomains $(a,0)$ and $(0,b)$. Following the proof of \cite[Lemma 9.1.2]{Jacob2012}, it can be shown that if $\gamma$ and $\tau$ are sufficiently large, then there is a $\kappa > 0$ such that
\begin{equation*}
    \frac{d}{dz} F(z) \geq - \kappa F(z), \quad z \in [a,b] \setminus \{ 0 \}. 
\end{equation*}
Applying Gr{\"o}nwall's Lemma (see, e.g., \cite[Section 2.4]{Teschl2012}), the following holds on the respective subdomain:
\begin{align}
\label{Lemma Auxiliary Lemma for Exponential Stability - Function F Estimates}
\begin{split}
    F(z) &\leq F(b) e^{\kappa (b-z)} ,\quad z \in (0, b], \\
    F(z) &\leq F(0^-) e^{-\kappa z}, \quad z \in [a,0).
    \end{split}
\end{align}
We want to show that there is some $\tilde{C} > 0$ such that for all $z \in [a,b] \setminus \{ 0 \}$ we have
\begin{align}
\label{Lemma Auxiliary Lemma for Exponential Stability - Function F Desired Estimate}
    F(z) \leq \tilde{C} F(b) e^{\kappa (b-a)}.
\end{align}
Using the notation (cf. \eqref{eq:Interface-variables-at-0}) $f_I:=f_{I, \mathcal{Q}_0 x} = \left( \mathcal{Q}^- x \right)_2 (0^-)$, continuity of $\mathcal{Q}^-$, as well as the fact that 
$\mathcal{Q}^-(0)$ is symmetric, we obtain that 
	\begin{align*}
		F(0^-) 
		&= \int_{\gamma b}^{\tau - \gamma b} x^{\top}(0^-,t) \mathcal{Q}^-(0) x(0^-,t) \, dt \\
		&= \int_{\gamma b}^{\tau - \gamma b} \begin{bmatrix}
			\left( \mathcal{Q}^-(0) x(0^-,t) \right)_1 \\
			f_I
		\end{bmatrix}^{\top} \left( \mathcal{Q}^-(0) \right)^{-1}  \begin{bmatrix}
		\left( \mathcal{Q}^-(0) x(0^-,t) \right)_1 \\
		f_I
	\end{bmatrix} \, dt.
	\end{align*}
As $m I \leq \mathcal{Q}^{\pm}(z) \leq MI$ for all $z \in [a,b]$, it follows for all $t \geq 0$ that
\begin{align*}
	\frac{1}{M} \int_{\gamma b}^{\tau - \gamma b} \left| \begin{bmatrix}
		\left( \mathcal{Q}^-(0) x(0^-,t) \right)_1 \\
		f_I
	\end{bmatrix} \right|^2 \, dt \leq F(0^-) \leq \frac{1}{m} \int_{\gamma b}^{\tau - \gamma b} \left| \begin{bmatrix}
	\left( \mathcal{Q}^-(0) x(0^-,t) \right)_1 \\
	f_I
\end{bmatrix} \right|^2 \, dt.
\end{align*}
Analogously, using $f_I=\left( \mathcal{Q}^+ x \right)_2 (0^+)$, see \eqref{eq:Interface-variables-at-0},
\begin{align*}
	\frac{1}{M} \int_{\gamma b}^{\tau - \gamma b} \left| \begin{bmatrix}
		\left( \mathcal{Q}^+(0) x(0^+,t) \right)_1 \\
		f_I
	\end{bmatrix} \right|^2 \, dt \leq F(0^+) \leq \frac{1}{m} \int_{\gamma b}^{\tau - \gamma b} \left| \begin{bmatrix}
		\left( \mathcal{Q}^+(0) x(0^+,t) \right)_1 \\
		f_I
	\end{bmatrix} \right|^2 \, dt.
\end{align*}
For $t \geq 0$, define $\alpha (t) := \begin{bmatrix}
    \left( \mathcal{Q}^-(0) x(0^-,t) \right)_1  & \left( \mathcal{Q}^+(0) x(0^+,t) \right)_1
\end{bmatrix}^\top$.  Then the interface condition \eqref{Interface Passivity Relation} can be rewritten as
	\begin{equation*}
		f_I = r e_I =  r \Big(\left( \mathcal{Q}^-(0) x(0^-,t) \right)_1 - \left( \mathcal{Q}^+(0) x(0^+,t) \right)_1 \Big)
		=  r \left( \alpha_1(t) - \alpha_2(t) \right).  
	\end{equation*}
Now consider the quadratic norms $|\cdot|_{V_1}$ and $|\cdot|_{V_2}$ defined via the quadratic forms given by the symmetric, positive definite matrices
\begin{equation*}
    V_1 := \begin{bmatrix}
     1+r^2 & -r^2\\ - r^2 & r^2
 \end{bmatrix}\,, \quad     V_2:= \begin{bmatrix}
     r^2 & -r^2\\ -r^2 & 1+r^2
 \end{bmatrix},
\end{equation*}
and let $k_1\geq 1$ be such that always $|\alpha|^2_{V_1}\leq k_1 |\alpha|^2_{V_2}$.
With this we have 
\begin{multline*}
  F(0^-)  \leq \frac{1}{m} \int_{\gamma b}^{\tau - \gamma b} \alpha_1(t)^2 + r^2 (\alpha_1(t) - \alpha_2(t))^2 \, dt = \frac{1}{m} \int_{\gamma b}^{\tau - \gamma b} \left| 
       \alpha(t) 
   \right|^2_{V_1} \, dt   \\
   \leq k_1 \frac{M}{m} \frac{1}{M} \int_{\gamma b}^{\tau - \gamma b} \left| 
       \alpha(t)  \right|^2_{V_2} \, dt 
       \leq k_1 \frac{M}{m} F(0^+).
\end{multline*}

By virtue of \eqref{Lemma Auxiliary Lemma for Exponential Stability - Function F Estimates}, simple calculations yield that for all $z \in [a,0)$ it holds
\begin{align*}
    F(z) \leq F(0^-) e^{-\kappa z} \leq k_1 \frac{M}{m} F(0^+) e^{- \kappa a} \leq k_1 \frac{M}{m} F(b) e^{\kappa (b-a)},
\end{align*}
and for all $z \in (0,b]$, we have
\begin{align*}
    F(z) \leq F(b) e^{\kappa (b-z)} \leq k_1 \frac{M}{m} F(b) e^{\kappa (b-a)}.
\end{align*}
Choosing $\tilde{C} = k_1 \frac{M}{m} $, we conclude that \eqref{Lemma Auxiliary Lemma for Exponential Stability - Function F Desired Estimate} holds for all $z \in [a,b] \setminus \{ 0 \}$. 

Now, we are able to verify the first estimate in \eqref{Lemma Auxiliary Lemma for Exponential Stablity 2 - Inequality 2}. Following the proof of \cite[Lemma 9.1.2]{Jacob2012} and using the estimate \eqref{Lemma Auxiliary Lemma for Exponential Stability - Function F Desired Estimate}, it may be shown that
		\begin{align*}
			2(\tau - 2 \gamma (b-a))  \|x(\tau)\|_{\mathcal{Q}_0}^2 \leq \frac{(b-a) \tilde{C}}{m} e^{\kappa (b-a)}  \int_{0}^{\tau} | \mathcal{Q}^+(b) x(b,t) |^2 \, dt.
			\end{align*}
		Letting $C = \frac{(b-a) \tilde{C} e^{\kappa (b-a)} }{2m (\tau - 2 \gamma (b-a))}$, we conclude that
		the first inequality in \eqref{Lemma Auxiliary Lemma for Exponential Stablity 2 - Inequality 2} holds. The second estimate in \eqref{Lemma Auxiliary Lemma for Exponential Stablity 2 - Inequality 2} can be shown in a similar fashion by defining
\begin{align*}
	\tilde{F}(z) := \int_{\gamma (z - a)}^{\tau - \gamma (z - a)} x^{\top}(z,t) \mathcal{Q}_0(z) x(z,t) \, dt, \quad z \in [a,b] \setminus \{ 0\}.
\end{align*}
\qed
\end{proof}
\begin{remark}
    It can be shown that the uniform bound $k_1$ in the above proof has to be chosen such that $k_1 > \frac{r^2 + 1}{r^2}$.
\end{remark}
	Lemma \ref{Lemma Auxiliary Lemma for Exponential Stability} shows that there exists some time $\tau >0 $ such that the energy of a state can be bounded by the sum of the accumulated energies at both boundaries during the time interval $[0, \tau]$, and in the case $r > 0$, it can be bounded by the accumulated energy at either boundary. This result allows us to prove the following sufficient condition for exponential stability of the semigroup generated by $A_{\mathcal{Q}_0}$.

	\begin{theorem}
		\label{Theorem Exponential Stability}
		Let the assumptions \ref{MatrixQ0Assumption1}-\ref{ContractionAssumption} hold. Assume there exists some $k > 0$ such that 
		\begin{equation}
			\label{Theorem Exponential Stability - Inequality}
			\langle A_{\mathcal{Q}_0} x_0, x_0 \rangle_{\mathcal{Q}_0} \leq  - k \left( | (\mathcal{Q}_0 x_0)(a) |^2 + | (\mathcal{Q}_0 x_0)(b) |^2 \right)
		\end{equation}
		holds for all $x_0 \in D(A_{\mathcal{Q}_0})$. Then $A_{\mathcal{Q}_0}$ generates an exponentially stable $C_0$-semigroup. In addition, if $r > 0$, then exponential stability of the $C_0$-semigroup generated by $A_{\mathcal{Q}_0}$ follows if
  one of the following estimates holds for all $x_0 \in D(A_{\mathcal{Q}_0})$:
		 \begin{align*}
        \langle A_{\mathcal{Q}_0} x_0, x_0 \rangle_{\mathcal{Q}_0} &\leq - k | (\mathcal{Q}_0 x_0)(b) |^2, \\
        \langle A_{\mathcal{Q}_0} x_0, x_0 \rangle_{\mathcal{Q}_0} &\leq - k | (\mathcal{Q}_0 x_0)(a) |^2.
    \end{align*}
	\end{theorem}
	
	\begin{proof}
		This follows immediately from the energy estimates in Lemma \ref{Lemma Auxiliary Lemma for Exponential Stability} and from Theorem 9.1.3 in \cite{Jacob2012}.
        \qed
	\end{proof}
	Compared with \cite[Theorem 5.1]{Augner2020} the second statement of the previous theorem is a slight improvement, as in the case $r>0$ only a condition on one of the boundary points is required. Lastly, we want to present another sufficient condition for the exponential stability of the port-Hamiltonian system \eqref{Simplified PH-System}. To prove the following theorem, one follows the exposition of Lemma 9.1.4 in \cite{Jacob2012} and exploits the passivity relation at the interface, as elaborated in \cite[Theorem 5.19]{Kilian2022}. The proof relies on the fact that the estimate (9.25) in \cite{Jacob2012} has been shown with respect to the evaluation at both spatial boundaries.
	\begin{theorem}
		\label{Theorem 2 Exponential Stability}
		Let the assumptions \ref{MatrixQ0Assumption1}-\ref{ContractionAssumption} hold. If the matrix $W_B$ satisfies $W_B \Sigma_4 W_B^{\top} > 0$, then $A_{\mathcal{Q}_0}$ generates an exponentially stable $C_0$-semigroup. 
	\end{theorem}
	
	In this section, we have shown that if we impose certain boundary conditions, then we can ensure that the contraction semigroup $(T(t))_{t \geq 0}$ generated by the port-Hamiltonian operator $A_{\mathcal{Q}_0}$ is of type $C_0(1, -\alpha)$ for some $\alpha > 0$, i.e.,
	\begin{equation*}
		\|T(t)\|_{\mathcal{L}(X, \|\cdot\|_{\mathcal{Q}_0})} \leq e^{-\alpha t}, \quad t \geq 0. 
	\end{equation*}
	As opposed to the boundary conditions, the interface condition plays a minor role in the proof of Theorem \ref{Theorem 2 Exponential Stability}. It is only required that $r \geq 0$, that is, there is no power inflow at the interface position.

\section{Example: Coupled Acoustic Waveguides}
\label{sec:Example}
	
In this section, we show how our results can be applied to specific situations.
We consider the example of two acoustic waveguides coupled by an interface consisting of some membrane that acts as a local acoustic impedance as it appears, for instance, in \cite{Ogam_ApplAcoustics_21}. 

	Let $a<0<b$.
	Consider two acoustic waveguides on the spatial intervals $[a,0)$ and $(0,b]$, respectively. On each subdomain, the acoustic wave model is given by the coupled system of linear balance equations consisting of the mass balance equation and the momentum balance equation. One may choose as energy variables the condensation $s$, which is the relative change of mass density of the air, and the momentum $\pi$. Then the system of balance equations may be expressed on each subdomain by
		\begin{align*}
			\frac{\partial}{\partial t} \begin{bmatrix}
				s^{\pm}(z,t) \\
				\pi^{\pm}(z,t)
			\end{bmatrix} &= \begin{bmatrix}
				0 & -\frac{\partial}{\partial z} \\
				- \frac{\partial}{\partial z} & 0 
			\end{bmatrix} \begin{bmatrix}
				 p^{\pm}(z,t) \\
				v^{\pm}(z,t)
			\end{bmatrix},
		\end{align*}
		where $s^{\pm}(z,t) \in \mathbb{R}$ is the condensation and $p^{\pm}(z,t) \in \mathbb{R}$ is the relative pressure at time $t \geq 0$, and $\pi^{\pm}(z,t) \in \mathbb{R}$ denotes the momentum and $v^{\pm}(z,t) \in \mathbb{R}$ denotes the velocity at the positions $z \in [a,0)$ and $ z \in (0,b]$, respectively. The variables $e^{\pm} := (p^{\pm}, v^{\pm} )$ are related to the state variables $x^{\pm} := ( s^{\pm}, \pi^{\pm})$ by the linear map defined by the coercive multiplication operators $\mathcal{Q}^{\pm}$ as follows:
		\begin{equation*}
			e^{\pm}(\cdot) =(\mathcal{Q}^{\pm}x^{\pm})(\cdot) = \begin{bmatrix}
				B^{\pm}(\cdot) & 0 \\
				0 & \frac{1}{\rho_0^{\pm}(\cdot)}  \end{bmatrix} \begin{bmatrix}
				 s^{\pm}(\cdot) \\
				\pi^{\pm}(\cdot)
			\end{bmatrix}, \quad x ^{\pm} \in X^{\pm},
		\end{equation*}
  where $\rho_0^{\pm}(z) >0$ and $B^{\pm}(z) >0$ are the mean mass density and the adiabatic bulk modulus at $z$, respectively.  We assume that $\frac{1}{\rho_0^-}, B^- \in L^{\infty}((a,0), \mathbb{R})$, and $\frac{1}{\rho_0^+}, B^+ \in L^{\infty}((0,b), \mathbb{R})$.

Let $X^- := L^2((a,0), \mathbb{R}^2)$ and $X^+ := L^2((0,b), \mathbb{R}^2)$.  
The mechanical energy of the corresponding system is given by the following Hamiltonian functionals:
	\begin{align*}
		H^-(x^-) :=& \int_{a}^{0} \mathcal{H}^-(x^-(z)) \, dz = \frac{1}{2} \langle x^-, \mathcal{Q}^- x^- \rangle_{L^2(a,0)}, \quad x^- \in X^-, \\
			H^+(x^+) :=& \int_{0}^{b} \mathcal{H}^+(x^+(z)) \, dz = \frac{1}{2} \langle x^+, \mathcal{Q}^+ x^+ \rangle_{L^2(0,b)}, \quad x^+ \in X^+.
		\end{align*}
		
The two acoustic waveguides are coupled at $z = 0$ by an acoustic impedance, which we assume to be of resistive type with resistance $R_I >0$ . Furthermore, we set the (relative) pressure at $z = a$ to zero, and we put another resistive impedance with resistance $R_b > 0$ at $z = b$. 
		
 We want to show by means of our results that the coupled system is well-posed and exponentially stable. 

\textbf{Well-posedness.} 		
		The model of the coupling is precisely given by the system \eqref{Simplified PH-System} (with $l = 0$) defined on the state space $X := L^2([a,b], \mathbb{R}^2)$, where $[a,b]$ is the composed spatial domain. Defining the coercive matrix operator $\mathcal{Q}_0 := c_0^- \mathcal{Q}^- + c_0^+ \mathcal{Q}^+ \in \mathcal{L}(X)$, we may endow $X$ with the inner product $\langle \cdot , \cdot \rangle_{\mathcal{Q}_0}$ defined in \eqref{Inner Product wrt Q0}.
		
		Due to the resistive acoustic impedance placed at the interface, we expect a pressure drop, and therefore have a discontinuity of the pressure $p := c_0^- p^- + c_0^+ p^+$, whereas the velocity is continuous at the interface $v := c_0^- v^- + c_0^+ v^+$. Thus, the interface port variables are given by
		\begin{align*}
			f_{I, \mathcal{Q}_0 x} :=& v(0^+) = v(0^-), \\
			-e_{I, \mathcal{Q}_0 x} :=& p(0^+) - p(0^-), 
		\end{align*}
		and are related by the passivity relation
		\begin{equation*}
			f_I(t) = \frac{1}{R_I} e_I(t), \quad t \geq 0.
		\end{equation*}
		Furthermore, we have the boundary conditions
		\begin{equation*}
			p(a,t) = 0 \quad \text{and} \quad p(b,t) = R_b v(b,t), \quad t \geq 0.
		\end{equation*}
		Using the definition of the boundary port variables  \eqref{Boundary Flow and Effort} (and neglecting the time-dependence), we may write equivalently
		\begin{equation*}
			0 = \begin{bmatrix}
				0 & 0 & 1 & 0 \\
				-1 & R_b & 0 & 0 
			\end{bmatrix} \begin{bmatrix}
				p(b) \\
				v(b) \\
				p(a) \\
				v(a)
			\end{bmatrix} = \tilde{W}_B \begin{bmatrix}
				(\mathcal{Q}_0x)(b) \\
				(\mathcal{Q}_0x)(a)
			\end{bmatrix} = W_B \begin{bmatrix}
				f_{\partial, \mathcal{Q}_0 x} \\
				e_{\partial, \mathcal{Q}_0 x}
			\end{bmatrix},
		\end{equation*}
		with $W_B := \tilde{W}_B R_{\extern}^{-1} = \frac{1}{\sqrt{2}} \begin{bmatrix}
			0 & 1 & 1 & 0 \\
			-R_b & 1 & -1 & R_b
		\end{bmatrix} \in \mathbb{R}^{2 \times 4}$. Note that $\rank(W_B) = 2$. Furthermore, one computes that
		\begin{equation}
			W_B \Sigma_4 W_B^{\top} = \begin{bmatrix}
				0 & 0 \\
				0 & 2R_b
			\end{bmatrix} \geq 0.
            \label{eq:example boundary condition}
		\end{equation}
		Summing up, the port-Hamiltonian operator $A_{\mathcal{Q}_0}$ with domain specified by $W_B$ defined above for the coupled acoustic waves, satisfies the assumptions of Theorem \ref{Theorem A Generates a Contraction Semigroup}. Hence, $A_{\mathcal{Q}_0}$ generates a contraction semigroup on the energy space $(X, \|\cdot\|_{\mathcal{Q}_0})$. In particular, this system is well-posed.

\textbf{Exponential stability.} 
Theorem~\ref{Theorem Exponential Stability} may be used to infer exponential stability. 
By virtue of the power balance \eqref{Power Balance Equation wrt Boundary and Interface Port}, the associated port-Hamiltonian operator $A_{\mathcal{Q}_0}$ satisfies
\begin{align*}
    2 \langle A_{\mathcal{Q}_0} x, x \rangle_{\mathcal{Q}_0} = \langle e_{\partial, \mathcal{Q}_0 x}, f_{\partial, \mathcal{Q}_0 x} \rangle - e_{I, \mathcal{Q}_0 x} f_{I, \mathcal{Q}_0 x} \leq  \langle e_{\partial, \mathcal{Q}_0 x}, f_{\partial, \mathcal{Q}_0 x} \rangle
\end{align*}
for all $x \in D(A_{\mathcal{Q}_0})$. Using the boundary conditions, simple computations yield that
\begin{align*}
      \langle e_{\partial, \mathcal{Q}_0 x}, f_{\partial, \mathcal{Q}_0 x} \rangle = - R_b v(b)^2.
\end{align*}
Furthermore, 
\begin{align*}
    | (\mathcal{Q}^+ x)(b) |^2 = p(b)^2 + v(b)^2 = (1 + R_b^2) v(b)^2. 
\end{align*}
In summary, we have 
\begin{align*}
     \langle A_{\mathcal{Q}_0} x, x \rangle_{\mathcal{Q}_0} \leq \frac{1}{2} \langle e_{\partial, \mathcal{Q}_0 x}, f_{\partial, \mathcal{Q}_0 x} \rangle = \frac{-R_b}{2(1 + R_b^2)} | (\mathcal{Q}^+ x)(b) |^2
\end{align*}
for all $x \in D(A_{\mathcal{Q}_0})$. Hence, by Theorem \ref{Theorem Exponential Stability}, this system is exponentially stable. 

Note that exponential stability is lost when the acoustic impedance $R_{I}$ at the interface tends to infinity: In this case, we have $r = \frac{1}{R_I} \to 0$ as $R_{I} \to \infty$. In particular, the subsystems are perfectly isolated from each other and the boundary $z = b$ is never met. This reflects the Dirichlet condition at $z = l = 0$ of the subsystem defined on the interval $(a,0)$. The acoustic impedance $R_b$ is never active and exponential stability is lost. From a technical perspective, the inequality \eqref{eq:example boundary condition} is not strict, and so Theorem \ref{Theorem 2 Exponential Stability} is not applicable. Finally, as 
\begin{equation*}
    | (\mathcal{Q}^-x)(a)|^2 = p(a)^2 + v(a)^2 = v(a)^2,
\end{equation*}
the inequality \eqref{Theorem Exponential Stability - Inequality} does not hold for all $x_0 \in D(A_{\mathcal{Q}_0})$ either, and Theorem~\ref{Theorem Exponential Stability} is not applicable as well.

\section{Port-Hamiltonian Systems with Multiple Interfaces}
\label{Section Port-Hamiltonian systems with multiple interfaces}

It is straightforward to show that the model and the results presented in this paper for a single stationary interface can be generalized to the interface coupling of an arbitrary amount of interfaces. This situation often occurs in multiphase or heterogeneous systems, for instance, as an extension of the example, an acoustic
wave travelling in cement-based materials \cite{Aggelis_07}.
%systems of conservation laws. 
Let $ \mathcal{I} = \{ l_1, l_2, \ldots, l_n \} \subset (a,b)$ be the set of the fixed positions of the interfaces satisfying $ l_1 < l_2 < \ldots < l_n$, separating $n+1$ systems of conservation laws. The corresponding port-Hamiltonian system is given by
	\begin{equation}
		\label{Simplified PH-System - multiple interfaces}
		\frac{\partial}{\partial t} x(z,t) = \mathcal{J}_{\mathcal{I}} \big( \mathcal{Q}_{\mathcal{I}}(z) x(z,t) \big), \quad z \in [a,b] \setminus \mathcal{I}, \quad t >0.
	\end{equation}
Here, $\mathcal{Q}_{\mathcal{I}} = \sum_{k=1}^{n+1} c^k \mathcal{Q}^k \in \mathcal{L}(X)$ is a coercive matrix multiplication operator, and the functions $c^k$, $k \in \{1, \ldots, n+1\}$, are the color functions defined as follows:
	\begin{align*}
		c^1(z) := c_a^{l_1}(z) &= \begin{cases} 
			1, & z \in [a,l_1), \\
			0, & \mathrm{else},
		\end{cases} \quad c^2(z) := c_{l_1}^{l_2}(z) =  \begin{cases} 
			1, & z \in (l_1,l_2), \\
			0, & \mathrm{else},
		\end{cases} \\
  \vdots \\
  c^n (z) := c_{l_{n-1}}^{l_n}(z) &=  \begin{cases} 
			1, & z \in (l_{n-1},l_n), \\
			0, & \mathrm{else},
   \end{cases} \quad 
  c^{n+1}(z) := c_{l_n}^{b}(z) =  \begin{cases} 
			1, & z \in (l_n,b], \\
			0, & \mathrm{else}.
   \end{cases}
	\end{align*}
	The operator $\mathcal{J}_{\mathcal{I}} \colon D(\mathcal{J}_{\mathcal{I}}) \subset X \to X$ is given by
	\begin{align*}
	%	\begin{split}
			D(\mathcal{J}_{\mathcal{I}}) &= \left\{ x = (x_1,x_2) \in X \mid x_1 \in D(\mathbf{d}_{\mathcal{I}}^{\ast}), \, x_2 \in D(\mathbf{d}_{\mathcal{I}})\right\}, \\
			\mathcal{J}_{\mathcal{I}} x &= \begin{bmatrix}
				0 & \mathbf{d}_{\mathcal{I}}  \\
				- \mathbf{d}_{\mathcal{I}}^{\ast} & 0 
			\end{bmatrix} \begin{bmatrix}
				x_1 \\
				x_2
			\end{bmatrix}
			=
			 \begin{bmatrix}
				\mathbf{d}_{\mathcal{I}} x_2 \\
				- \mathbf{d}_{\mathcal{I}}^{\ast} x_1
			\end{bmatrix}, \quad x \in D(\mathcal{J}_{\mathcal{I}}).
	%	\end{split}
	%	\label{Operator Jl}
	\end{align*}
	Here the operator $\mathbf{d}_{\mathcal{I}}$ is defined as follows:
 \begin{align*}
%\begin{split}
	 D(\mathbf{d}_{\mathcal{I}})
	 & = H^1([a,b], \mathbb{R}),  \\
	\mathbf{d}_{\mathcal{I}} x &= - \left[ \sum_{k=1}^{n+1} \frac{d}{dz} (c^k x) \right], \quad x \in D(\mathbf{d}_{\mathcal{I}}).
 %\end{split}
 %\label{Operator dl}
\end{align*}
Analogously to the single interface case, for a function $x = \sum_{k=1}^{n+1} c^k x^k \in D(\mathbf{d}_{\mathcal{I}})$ we have that
\begin{align*}
\mathbf{d}_{\mathcal{I}} x = - \sum_{k=1}^{n+1} c^k \frac{d}{dz} x^k.
\end{align*}
Its formal adjoint is given by
	\begin{align*}
		D(\mathbf{d}_{\mathcal{I}}^{\ast}) &= \left\{ y \in L^2([a,b], \mathbb{R}) \mid y_{|(a,l_1)} \in H^1((a,l_1), \mathbb{R}), \, y_{|(l_k,l_{k+1})} \in H^1((l_k,l_{k+1}), \mathbb{R}), \right. \\
  &\hspace{5 cm} \left. k \in \{1, \ldots, n-1\}, \,y_{|(l_n,b)} \in H^1((l_n,b), \mathbb{R}) \right\},  \\
		\mathbf{d}_{\mathcal{I}}^{\ast}y &= \left( - \mathbf{d}_{\mathcal{I}} - \left[ \sum_{k=1}^{n+1}\frac{d}{dz} c^k \right] \right) y, \quad y \in D(\mathbf{d}_{\mathcal{I}}^{\ast}).
	\end{align*}
Once again, for a function $y = \sum_{k=1}^{n+1} c^k y^k \in D(\mathbf{d}_{\mathcal{I}}^{\ast})$ it holds that
	\begin{equation*}
		\mathbf{d}_{\mathcal{I}}^{\ast} y = \sum_{k=1}^{n+1} c^k \frac{d}{dz} y^k.
	\end{equation*} 
	For all $x \in D(\mathbf{d}_{\mathcal{I}}), \,  y \in D(\mathbf{d}_{\mathcal{I}}^{\ast})$, we have the following relation between the operators $\mathbf{d}_{\mathcal{I}}$ and $\mathbf{d}_{\mathcal{I}}^{\ast}$ (cf. \eqref{Relation dl and dl_ast}):
	\begin{equation*}
	%	\label{Relation dl and dl_ast}
		\langle \mathbf{d}_{\mathcal{I}} x , y \rangle_{L^2} =  - \big[ x(z) y(z) \big]_{a}^{b} + \sum_{k =1}^{n} x(l_k) \left[ y(l_k^+) - y(l_k^-) \right] + \langle x , \mathbf{d}_{\mathcal{I}}^{\ast} y \rangle_{L^2}.
	\end{equation*} 
 Furthermore, for all $k \in \{1,\ldots, n\}$ and some fixed $r_k \in \mathbb{R}$, the interface conditions are given as
	\begin{equation*}
		%\label{Interface Passivity Relation}
		f_{I_k, \mathcal{Q}_{\mathcal{I}} x} = r_k e_{I_k, \mathcal{Q}_{\mathcal{I}} x},
	\end{equation*}
	where the interface port variables are defined by
	\begin{align*}
	%\label{eq:Interface-variables-at-0}
% \begin{split}
		f_{I_k, \mathcal{Q}_{\mathcal{I}} x} &= \left( \mathcal{Q}^{k+1}x\right)_2 (l_k^+) = \left( \mathcal{Q}^k x \right)_2 (l_k^-), \\
		-e_{I_k, \mathcal{Q}_{\mathcal{I}} x} &= \left( \mathcal{Q}^{k+1}x \right)_1 (l_k^+) - \left( \mathcal{Q}^k x \right)_1 (l_k^-).
 % \end{split}
    \end{align*}
	With the usual boundary conditions and these interface conditions, we may define the port-Hamiltonian operator $A_{\mathcal{Q}_{\mathcal{I}}} \colon D(A_{\mathcal{Q}_{\mathcal{I}}}) \subset X \to X$  associated with the system \eqref{Simplified PH-System - multiple interfaces} by
	\begin{align}
		\begin{split}
			D(A_{\mathcal{Q}_{\mathcal{I}}}) &= \left\{ x \in X \, \big| \, \mathcal{Q}_{\mathcal{I}} x \in D(\mathcal{J}_{\mathcal{I}}), \, f_{I_k, \mathcal{Q}_{\mathcal{I}}x} = r_k e_{I_k, \mathcal{Q}_{\mathcal{I}} x}, \, k \in \{1,\ldots, n\}, \right. \\
   &\hspace{7 cm} \left. W_B \begin{bmatrix}
				f_{\partial, \mathcal{Q}_{\mathcal{I}} x} \\
				e_{\partial, \mathcal{Q}_{\mathcal{I}} x}
			\end{bmatrix} = 0 \right\}, \\
			A_{\mathcal{Q}_{\mathcal{I}}}x &= \mathcal{J}_{\mathcal{I}} (\mathcal{Q}_{\mathcal{I}} x), \quad  x \in D(A_{\mathcal{Q}_{\mathcal{I}}}).
		\end{split}
		\label{Simplified PH System - Operator A - multiple interfaces}
	\end{align}
With some slight adjustments of the arguments presented in Sections \ref{Section Generation of a Contraction Semigroup} and \ref{Section Exponential Stability}, one can characterize the boundary and interface conditions under which the operator $A_{\mathcal{Q}_{\mathcal{I}}}$ generates a contraction semigroup (in accordance with Theorem \ref{Theorem A Generates a Contraction Semigroup}), and find sufficient conditions for the exponential stability of said semigroup (in accordance with the second assertion in Theorem \ref{Theorem Exponential Stability}).

\begin{corollary}
\label{Theorem A Generates a Contraction Semigroup - Multiple Interfaces}
    Consider the port-Hamiltonian operator $A_{\mathcal{Q}_{\mathcal{I}}} \colon D(A_{\mathcal
		Q_{\mathcal{I}}}) \subset X \to X$ defined in \eqref{Simplified PH System - Operator A - multiple interfaces} with $r_k \in \mathbb{R}$ for all $k \in \{1, \ldots, n\}$, and with a full rank matrix $W_B \in \mathbb{R}^{2 \times 4}$. Then the following assertions are equivalent:
		\begin{enumerate}
		    \item $A_{\mathcal{Q}_{\mathcal{I}}}$ is the infinitesimal generator of a contraction semigroup on $X$ endowed with the inner product 
		    $\langle \cdot, \cdot \rangle_{\mathcal{Q}_{\mathcal{I}}} = \frac{1}{2} \langle \cdot , \mathcal{Q}_{\mathcal{I}} \cdot \rangle_{L^2}$.
		    \item $A_{\mathcal{Q}_{\mathcal{I}}}$ is dissipative.
		    \item $r_k \geq 0$ for all $k \in \{1, \ldots, n\}$, and with $\Sigma_4$ defined in \eqref{Matrix Sigma}
		    we have
		    \begin{equation}
		        W_B \Sigma_4 W_B^{\top} \geq 0.
		        \label{Matrix Condition W_B_Multiple Interfaces}
		    \end{equation}
		\end{enumerate}  
        In addition, $A_{\mathcal{Q}_{\mathcal{I}}}$ is the infinitesimal generator of an isometric semigroup on the energy space $(X, \| \cdot \|_{\mathcal{Q}_{\mathcal{I}}})$ if and only if $r_k = 0$ for all $k \in \{1, \ldots, n\}$ and $W_B$ satisfies $W_B \Sigma_4 W_B^{\top} = 0$.
\end{corollary}

For exponential stability of the generated semigroup, we need to impose the following assumptions, which are slightly stricter than \ref{MatrixQ0Assumption1}-\ref{ContractionAssumption}:
	\begin{enumerate}[label = (B\arabic*)]
		\item \label{MatrixQ0Assumption1_Multiple Interfaces} The matrix operators $\mathcal{Q}^{k}$, $k \in \{1, \ldots, n+1\}$, defining the coercive operator $\mathcal{Q}_{\mathcal{I}} \in \mathcal{L}(X)$ satisfy $\mathcal{Q}^{k} \in \mathcal{C}^1([a,b], \mathbb{R}^{2 \times 2})$. 
		\item \label{ContractionAssumption_Multiple Interfaces} $r_k > 0$ for all $k \in \{1, \ldots, n\}$, $\rank(W_B) = 2$, and $W_B$ satisfies \eqref{Matrix Condition W_B_Multiple Interfaces}. 
	\end{enumerate}
    
In other words, Assumption \ref{ContractionAssumption_Multiple Interfaces} requires that energy is dissipated at each interface position, and that the port-Hamiltonian operator $A_{\mathcal{Q}_{\mathcal{I}}}$ generates a contraction semigroup (see Corollary~\ref{Theorem A Generates a Contraction Semigroup - Multiple Interfaces}).\begin{corollary}
    Let the assumptions \ref{MatrixQ0Assumption1_Multiple Interfaces}-\ref{ContractionAssumption_Multiple Interfaces} hold. Assume there exists some $k > 0$ such that either of the following estimates holds for all $x_0 \in D(A_{\mathcal{Q}_{\mathcal{I}}})$:
		 \begin{align*}
        \langle A_{\mathcal{Q}_{\mathcal{I}}} x_0, x_0 \rangle_{\mathcal{Q}_{\mathcal{I}}} &\leq - k | (\mathcal{Q}_{\mathcal{I}} x_0)(b) |^2, \\
        \langle A_{\mathcal{Q}_{\mathcal{I}}} x_0, x_0 \rangle_{\mathcal{Q}_{\mathcal{I}}} &\leq - k | (\mathcal{Q}_{\mathcal{I}} x_0)(a) |^2.
    \end{align*}
    Then $A_{\mathcal{Q}_{\mathcal{I}}}$ generates an exponentially stable $C_0$-semigroup on $(X, \| \cdot \|_{\mathcal{Q}_{\mathcal{I}}})$.
\end{corollary}

	\section{Conclusions and Extensions Towards Moving Interfaces}
	\label{sec:conclusions}

	Based on the port-Hamiltonian model of systems with moving interfaces suggested in \cite{Diagne2013}, we have given a mathematically rigorous definition of the linear port-Hamiltonian operator $A_{\mathcal{Q}_l}$, defined with respect to the constant color functions \eqref{Characteristic Functions c- and c+}, appearing in such models.  For this choice of color functions, we have derived the balance equation of the Hamiltonian functional in terms of boundary variables and interface variables at the boundary of the spatial domain and at the interface.

    For the case of a fixed interface position, 
    %a set of tools is provided that fully characterizes contractivity and gives conditions for the exponential stability of such systems. 
    we have worked out necessary and sufficient conditions that guarantee that the associated port-Hamiltonian operator generates a contraction semigroup on the respective energy space. These passivity conditions have been defined with respect to the interface and boundary port variables and are easy to check in practice. Moreover, we have presented criteria for exponential stability of the $C_0$-semigroup generated by the associated port-Hamiltonian operator.

    Our study can be seen as a first step towards the development of the theory for the case of moving interfaces. At the time of writing this, there are still a number of notable obstacles towards this goal. It is tempting to let $l(\cdot)$ as a function of $t$ denote the moving interface and to consider the operators $A_{\mathcal{Q}_{l(t)}}$ as a family of operators that should generate an evolution family. Standard tools in this direction fail immediately, most notably because there is no subspace $Y$ dense in $L^2([a,b], \R^2)$ in the intersection of the domains $D(A_{\mathcal{Q}_{l(t)}})$, compare, e.g., \cite[Chapter 5]{Pazy1983}. 
    A route towards the generation of evolution families could be
provided, as it has been suggested in \cite{Diagne2013}, by considering
an additional pair of port variables associated with the displacement
of the interface. Thereby, a balance equation may be derived for the
Hamiltonian function and passivity properties may be established and
used. A very special case of this problem has been discussed in \cite{KMMW23}.
However, in this case, more general dissipative relations and possibly
dynamical systems should be considered at the interface, as they
appear, for instance, in the example of a piston moving in a chamber
and separating two gases \cite[Example 8]{Diagne2013}. It should also be noted that the same Hamiltonian operator (\ref{Operator Jl})
would appear in the case of diffusive systems, in the same way as
for systems without interface \cite[Chapter 6]{Villegas2007}, thereby
enabling us to consider the two--phase Stefan problem \cite[Chapter 4]{Visintin1996models},
used to model melting/solidification or vaporization/condensation
processes.

    %Clearly, the energy flow generated by the moving interface has to be taken into account here, as it has been done in  \cite{Diagne2013}, by defining an additional pair of port variables. A route towards the generation of evolution families could be provided by using these additional port variables and using passivity properties. A very special case of this problem has been discussed in \cite{KMMW23}.
	
%	\appendix
	
\section*{Acknowledgments}
The work of A. Kilian was supported
 by BMBF under Grant 16ME0619.
 A. Mironchenko was supported by DFG (German Research Foundation) grant MI 1886/2-2.

%\appendix

%\section{Sample Appendix Section}
%Lorem ipsum dolor sit amet, consectetur adipiscing elit, sed do eiusmod tempor section \ref{sec:sample1} incididunt ut labore et dolore magna aliqua. Ut enim ad minim veniam, quis nostrud exercitation ullamco laboris nisi ut aliquip ex ea commodo consequat. Duis aute irure dolor in reprehenderit in voluptate velit esse cillum dolore eu fugiat nulla pariatur. Excepteur sint occaecat cupidatat non proident, sunt in culpa qui officia deserunt mollit anim id est laborum.

 \bibliographystyle{elsarticle-num} 
\bibliography{references}





% \end{thebibliography}
\end{document}